\documentclass[a4paper,11pt,reqno]{amsart}
\usepackage[utf8]{inputenc}
\usepackage{amsaddr}
\usepackage{calrsfs}
\usepackage{graphicx,color}
\usepackage{subfig}

\numberwithin{equation}{section}

\theoremstyle{definition}
\newtheorem{dfn}{Definition}[section]
  
\theoremstyle{plain}
\newtheorem{thm}{Theorem}[section]
\newtheorem{pro}{Proposition}[section]

\theoremstyle{remark}
\newtheorem{rem}{Remark}[section]
\newtheorem{exa}{Example}[section]

\newcommand{\R}{\mathbb{R}}
\newcommand{\E}{\mathbb{E}}
\newcommand{\Var}{\mathbb{V}}
\renewcommand{\P}{\mathbb{P}}
\newcommand{\p}{\partial}
\newcommand\law{\mathrel{\overset{\makebox[0pt]{\mbox{\normalfont\tiny\sffamily law}}}{=}}}

\newcommand{\F}{\mathcal{F}}

\newcommand{\e}{\mathrm{e}}
\newcommand{\la}{\langle}
\newcommand{\ra}{\rangle}
\renewcommand{\H}{\mathcal{H}}
\newcommand{\1}{\mathbf{1}}
\renewcommand{\d}{\mathrm{d}}
\renewcommand{\t}{\mathbf{t}}

\begin{document}
\title[$n$th order fBm]
{Transfer Principle for $n$th order Fractional Brownian Motion with Applications to Prediction and Equivalence in Law}

\date{\today}

\author[Sottinen]{Tommi Sottinen}
\address{Department of Mathematics and Statistics, University of Vaasa, P.O. Box 700, FIN-65101 Vaasa, FINLAND}
\email{tommi.sottinen@iki.fi}

\author[Viitasaari]{Lauri Viitasaari}
\address{Department of Mathematics and Statistics, University of Helsinki, Helsinki, P.O. Box 68, FIN-00014 University of Helsinki,  FINLAND} 
\email{lauri.viitasaari@iki.fi}

\begin{abstract}
The $n$th order fractional Brownian motion was introduced by Perrin et al. \cite{Perrin-Harba-Berzin-Joseph-Iribarren-Bonami-2001}. It is the (upto a multiplicative constant) unique self-similar Gaussian process with Hurst index $H \in (n-1,n)$, having $n$th order stationary increments. We provide a transfer principle for the $n$th order fractional Brownian motion, i.e., we construct a Brownian motion from the $n$the order fractional Brownian motion and then represent the $n$the order fractional Brownian motion by using the Brownian motion in a non-anticipative way so that the filtrations of the $n$the order fractional Brownian motion and the associated Brownian motion coincide. By using this transfer principle, we provide the prediction formula for the $n$the order fractional Brownian motion and also a representation formula for all the Gaussian processes that are equivalent in law to the $n$th order fractional Brownian motion.
\end{abstract}




\maketitle


\section{Introduction}
The fractional Brownian motion is probably the most well-known generalisation of the Brownian motion. It is a centred Gaussian process $B_H$, depending on the Hurst parameter $H\in(0,1)$, that is $H$-self-similar and has stationary increments. In the centred Gaussian case these two properties characterize the process completely (upto a multiplicative constant). Nowadays the properties and stochastic analysis with respect to the fractional Brownian motion are studied extensively and widely understood. We refer to Mishura \cite{Mishura-2008} for details of fractional Brownian motion and its stochastic analysis.

The $n$th order fractional Brownian motion was introduced by Perrin et al. \cite{Perrin-Harba-Berzin-Joseph-Iribarren-Bonami-2001}. Their motivation was to extend the self-similarity index $H$ of the standard fractional Brownian motion beyond the limitation $H<1$. That is, the $n$th order fractional Brownian motion $B_H^{(n)}$ with Hurst parameter $H\in(n-1,n)$ is $H$-self-similar, and in the case $n=1$ one recovers the fractional Brownian motion. Moreover, the $n$th order fractional Brownian motion is $n$-stationary, meaning that its $n$th differences are stationary. As in the case of the fractional Brownian motion, these two properties: $H$-self-similarity and $n$-stationarity characterize the process in the centred Gaussian world (upto a multiplicative constant). 

Perrin et al. \cite{Perrin-Harba-Berzin-Joseph-Iribarren-Bonami-2001} defined the $n$th order fractional Brownian motion by using the so-called Mandelbrot--Van Ness representation \cite{Mandelbrot-van-Ness-1968} of the fractional Brownian motion. However, while the kernel in the Mandelbrot--Van Ness representation is rather simple, the representation requires the knowledge to the infinite past of the generating Brownian motion, which is very problematic in many applications. In this article, we propose to define the $n$th order fractional Brownian motion by using the Molchan--Golosov representation \cite{Molchan-Golosov-1969} of the fractional Brownian motion. As such, we obtain a compact interval representation with the minor cost of having a slightly more complicated kernel. In addition, we show that the filtrations of the $n$th order fractional Brownian motion and the standard Brownian motion generating it coincide. This is very important for filtering and prediction. In addition, we provide a transfer principle for the $n$th order fractional Brownian motion which can be used to develop stochastic calculus in an elementary and simple way by using the stochastic calculus with respect to the standard Brownian motion.

The rest of the paper is organised as follows. In Section \ref{sec:prel} we recall some elementary definitions and preliminaries. We also recall the transfer principle with respect to the fractional Brownian motion. In section \ref{sec:nfbm} we present our main results,  which we apply to the equivalence of laws problem and prediction problem in Section \ref{sec:applications}. We end the paper with a simulation study in Section \ref{sec:simu} and some concluding remarks in Section \ref{sec:conclusion}. 

\section{Preliminaries}
\label{sec:prel}
We begin by recalling the concept of $H$-self-similarity. 

\begin{dfn}
Let $X = (X_t)_{t\geq 0}$ be a stochastic process and $H>0$. The stochastic process $X$ is called $H$-self-similar if for any $a>0$ we have
$$
X_{at} \law a^{H}X_t, \quad t\ge 0,
$$
where $\law$ denotes the equality of the finite dimensional distributions.
\end{dfn}

\begin{dfn}
Denote an increment of size $l$ of a function $g$ by
$$
\Delta_l g(x) = g(x+l)-g(x).
$$
For $k\geq 1$, the $k$th order increment of size $l$ is defined recursively by
$$
\Delta_l^{(k)} g(x) = \Delta_l^{(k-1)}g(x+l)-\Delta_l^{(k-1)}g(x).
$$
where the $0$th-order increment is $\Delta_l^{(0)}g(x)=g(x)$.
\end{dfn}

\begin{exa}
The $2$nd order increment of size $l$, $\Delta_l^{(2)}g(x)$, is given by 
$$
\Delta_l^{(2)}g(x) = g(x+2l)-2g(x+l)+g(x)
$$
which corresponds to the usual second order increment. 
\end{exa}

\begin{dfn}
A stochastic process $X$ is called $n$-stationary if for every $l>0$ the process $\Delta_l^{(k)}X$ is non-stationary for $k=0,1,\ldots,n-1$, but the process $\Delta_l^{(n)}X$ is stationary.
\end{dfn}
\begin{rem}
Note that the definition of $n$-stationary is simply a continuous time analogue of the classical notion of difference stationarity, used e.g. in time series analysis and ARIMA-processes \cite{bro-dav}. 
\end{rem}

Recall that the fractional Brownian motion $B_H$ with Hurst index $H\in(0,1)$ is a centred Gaussian process having the covariance function 
\begin{equation}
\label{eq:cov_fbm}
r_H(t,s) = \frac{1}{2}\left[|t|^{2H} + |s|^{2H} - |t-s|^{2H}\right], \quad t,s\in\R.
\end{equation}

It is easy to check that the process $B_H$ is not stationary, but has stationary increments, i.e., the fractional Brownian motion is $1$-stationary.  It is also easy to check that the process $B_H$ is $H$-self-similar, and that the valid range for $H$ is $H\in(0,1]$. The case $H=1$ is obviously a degenerate one: $B_H(t) = t B_H(1)$.

The fractional Brownian motion $B_H$ is connected to the standard Brownian motion $W$ via integral representations. We begin by recalling the Mandelbrot--Van Ness representation \cite{Mandelbrot-van-Ness-1968}  of the fractional Brownian motion $B_H$ on the real line with Hurst index $H\in(0,1)$: 
\begin{equation}\label{eq:mvn}
B_H(t) =
\int_{-\infty}^t g_H(t,u) \, \d W(u), \quad t\in\R,
\end{equation}
where $W$ is a standard Brownian motion and the kernel is
$$
g_H(t,u)
=
\frac{1}{\Gamma(H+\frac12)}\left[(t-u)^{H-\frac12} - (-u)^{H-\frac12}_+\right].
$$
Here $x_+ := x\vee 0 := \max(x,0)$ and $\Gamma$ denotes the Gamma function:
$$
\Gamma(x) = \int_0^\infty t^{x-1} \e^{-t}\, \d t.
$$ 

\begin{rem}
We remark that \eqref{eq:mvn} defines the fractional Brownian motion on the whole real line, but the filtrations of the fractional Brownian motion $B_H$ and the generating Brownian motion $W$ coincides only for $t>0$. This is problematic for many applications.
\end{rem}

For many practical applications it is useful to represent a process as an integral with respect to a Brownian motion on a compact interval. In the case of the fractional Brownian motion, one can use the so-called invertible Molchan-Golosov representation (see, for example,  \cite{Norros-Valkeila-Virtamo-1999,Pipiras-Taqqu-2001}).
\begin{pro}[\cite{Molchan-2003,Molchan-Golosov-1969}]
\label{pro:fbm_MG}
The Molchan--Golosov representation of the fractional Brownian motion $B_H$ with $H\in(0,1)$ is
\begin{equation}\label{eq:mg}
B_H(t) = \int_0^t k_H(t,s)\, \d W(s), \quad t\ge 0,
\end{equation}
where
\begin{eqnarray}
\lefteqn{k_H(t,s) }\nonumber \\
\label{eq:kernel}
&=& d_H\left[\left(\frac{t}{s}\right)^{H-\frac{1}{2}}(t-s)^{H-\frac12}
-\left(H-\frac{1}{2}\right)s^{\frac{1}{2}-H}\int_s^t z^{H-\frac{3}{2}}(z-s)^{H-\frac{1}{2}}\, \d z\right],  
\end{eqnarray}
with normalising constant
$$
d_H = \sqrt{\frac{2H\Gamma(\frac{3}{2}-H)}{\Gamma(H+\frac{1}{2})\Gamma(2-2H)}}.
$$
Moreover, the representation \eqref{eq:mg} is invertible, i.e., the Brownian motion $W$ is constructed from the fractional Brownian motion $B_H$ by
\begin{equation}\label{eq:mg-inv}
W(t) = \int_0^t k^{(-1)}_H(t,u)\, \d B_H(u),
\end{equation}
where
\begin{eqnarray*}
\lefteqn{k^{(-1)}_H(t,u)}\\
&=& d_H' \left[\left(\frac{t}{s}\right)^{H-\frac{1}{2}}(t-s)^{\frac12-H}
-\left(H-\frac{1}{2}\right)s^{\frac{1}{2}-H}\int_s^t z^{H-\frac{3}{2}}(z-s)^{H-\frac{1}{2}}\, \d z\right],
\end{eqnarray*}
with normalising constant
$$
d_H' = \frac{\Gamma(H+\frac{1}{2})\Gamma(2-2H)}{\mathrm{B}\left(\frac12-H,H+\frac12\right)\sqrt{\left(2H+\frac12\right)\Gamma(\frac{1}{2}-H)}}.
$$
Here $\mathrm{B}$ denotes the Beta function:
$$
\mathrm{B}(a,b) = \frac{\Gamma(a)\Gamma(b)}{\Gamma(a+b)}.
$$
\end{pro}

\begin{rem}
We emphasise the invertible role of the Molchan--Golosov representation.  We have that the filtrations $\F_W(t)=\sigma\{ W(s), s\le t\}$ and $\F_{B_H}(t)=\sigma\{ B_H(s), s\le t\}$ coincide.  This is of paramount importance in, for example, prediction and filtering.  
\end{rem}

\begin{rem}\label{rem:mg-inv}
The Wiener integral in \eqref{eq:mg-inv} can be understood as a Riemann--Stieltjes integral, see \cite{Norros-Valkeila-Virtamo-1999}. More generally it can be understood by using a so-called transfer principle, as will be explained the the next subsection below.
\end{rem}

\subsection{Transfer principle for the fractional Brownian motion}
In this subsection we recall the Wiener integral and the transfer principle for the fractional Brownian motion. For this purposes we consider the fractional Brownian motion $B_H$ on some compact interval $[0,T]$. 


\begin{dfn}[Isonormal process]\label{dfn:isonormal}
The isonormal process associated with the fractional Brownian motion $B_H = ( B_H(t), t\in [0,T])$ with the Hurst index $H\in(0,1)$ is the centred Gaussian family $(B_H(h), h\in\H_H)$, where the Hilbert space $\H_H = \H_H([0,T])$ is generated by the covariance $r_H$ given by \eqref{eq:cov_fbm} as follows:
\begin{enumerate}
\item indicators $\1_t := \1_{[0,t)}$, $t\le T$, belong to $\H_H$.
\item $\H_H$ is endowed with the inner product $\la \1_t,\1_s\ra_{\H_H} := r_H(t,s)$, 
\end{enumerate}
and the centred Gaussian family is then defined by the covariance
$$
\E\left[B_H(h) B_H(\tilde h)\right] = \la h, \tilde h\ra_{\H_H}. 
$$
\end{dfn}

Definition \ref{dfn:isonormal} states that $B_H(h)$ is the image of $h\in\H_H$ in the isometry that extends the relation
$$
B_H\left(\1_t\right) := B_H(t)
$$
linearly. This gives rise to the definition of Wiener integral with respect to the fractional Brownian motion.

\begin{dfn}[Wiener integral]\label{dfn:wi}
$B_H(h)$ is the \emph{Wiener integral} of the element $h\in\H_H$ with respect to $B_H$, and it is denoted by
$$
\int_0^T h(t)\, \d B_H(t).
$$   
\end{dfn}
\begin{rem}
Due to the completion under the inner product $\la\cdot,\cdot\ra_{\H_H}$ it may happen that the space $\H_H$ is not a space of functions, but contains distributions. Indeed, it was shown by Pipiras and Taqqu \cite{Pipiras-Taqqu-2001} that the space $\H_H$ is a proper function space only in the case $H\leq \frac12$, while it contains distributions for $H>\frac12$. We also note that in the case of a standard Brownian motion, that is $H=\frac12$, we have $\H_{\frac12} = L^2([0,T])$.
\end{rem}

The kernel representation of the fractional Brownian motion can be used to provide a transfer principle for Wiener integrals. In other words, the Wiener integrals with respect to the fractional Brownian motion can be transferred to a Wiener integrals with respect to the corresponding Brownian 
motion. To do this,  we define a dual operator $k_H^*$ on $L^2([0,T])$ associated with the kernel $k_H$ by extending the relation
\begin{equation}\label{eq:kstar-distr}
k_H^*\1_t = k_H(t,\cdot)
\end{equation}
linearly and closing it in $\H_H$. 

\begin{rem}
For general $H\in(0,1)$ and step functions $f$ we have
\begin{equation}\label{eq:kstar}
(k_H^* f)(t) = k_H(T,t)f(t) + \int_t^T \left[ f(u) -f(t)\right]\, \frac{\partial k_H(u, t)}{\partial u}\, \d u.
\end{equation}
If, moreover, $H\in(\frac12,1)$, then \eqref{eq:kstar} can be simplified into
$$
(k_H^* f)(t) = \int_t^T f(u)\, \frac{\partial k_H(u, t)}{\partial u}\, \d u,
$$
since for $H>\frac12$ we have $k_H(t,t)=0$.
\end{rem}

\begin{rem}
The dual operator $k_H^*$ depends on the interval $[0,T]$, even though the kernel $k_H$ does not.  This is the reason we have to consider the Hilbert space $\H_H = \H_H([0,T])$.
\end{rem}

\begin{thm}[\cite{Alos-Mazet-Nualart-2001}]
\label{thm:fbm_transfer}
Let $k_H^*$ be given by \eqref{eq:kstar-distr}. Then $k_H^*$ provides an isometry between $\H_H$ and $L^2([0,T])$. Moreover, for any $f \in \H_H$ we have
$$
\int_0^T f(u)\, \d B_H(u) = \int_0^T (k^*_H f)(u)\, \d W(u).
$$
\end{thm}

\section{The $n$th order fractional Brownian motion}
\label{sec:nfbm}
Perrin et al. \cite{Perrin-Harba-Berzin-Joseph-Iribarren-Bonami-2001} defined the $n$th order fractional Brownian motion by using the Mandelbrot--Van Ness \cite{Mandelbrot-van-Ness-1968} representation of the fractional Brownian motion $B_H$ with Hurst index $H\in(0,1)$:

\begin{dfn}[\cite{Perrin-Harba-Berzin-Joseph-Iribarren-Bonami-2001}]\label{dfn:nfbm-mvn}
Let $W=(W(u), u\in\R)$ be the two-sided standard Brownian motion and let $\Gamma$ be the gamma function. The $n$th order fractional Brownian motion with Hurst index $H\in(n-1,n)$ is defined as
\begin{equation}\label{eq:mvn-nfbm}
B^{(n)}_{H}(t) =
\int_{-\infty}^t g_H^{(n)}(t,u)\, \d W(u), \quad t\in\R,
\end{equation}
where
\begin{eqnarray}
g_H^{(n)}(t,u) &=&
\frac{1}{\Gamma(H+\frac12)}
\left[ (t-u)^{H-\frac12}-(-u)_+^{H-\frac12} -\cdots-\right.  \nonumber \\
& & \left.
\left(H-\frac12\right)\cdots\left(H-\frac{2n-3}{2}\right)
(-u)_+^{H-n+\frac12}\frac{t^{(n-1)}}{(n-1)!}\right]. \label{eq:gn}
\end{eqnarray}
\end{dfn}

\begin{rem}
In the case $n=1$, the classical fractional Brownian motion is recovered, as Definition \ref{dfn:nfbm-mvn} reduces to the Mandelbrot--Van Ness representation of the fractional Brownian motion.
\end{rem}

\begin{rem}\label{rem:properties}
Some properties of the $n$th order fractional Brownian motion provided in \cite{Perrin-Harba-Berzin-Joseph-Iribarren-Bonami-2001} are 
\begin{enumerate}
\item $B_H^{(n)}$ is $n$-stationary,
\item $B_H^{(n)}$ is $H$-self-similar,
\item $B_H^{(n)}$ has $n-1$ times continuously differentiable paths. In particular, 
\begin{equation}\label{eq:diff-nfbm}
\frac{\d}{\d t} B^{(n)}_H(t) = B^{(n-1)}_{H-1}(t).
\end{equation}
\end{enumerate}
Properties (i) and (ii) are unique to the $n$th order fractional Brownian motion in the class of centred Gaussian processes, i.e., they can be used as a qualitative definition.  Property (iii) follows from the properties (i) and (ii).
\end{rem}

The covariance function $r_H^{(n)}(t,s) = \E[B^{(n)}_H(t) B^{(n)}_H(s)]$ of the $n$th order fractional Brownian motion is 
\begin{eqnarray*}
\lefteqn{r_H^{(n)}(t,s)} \nonumber \\
&=&
\frac{(-1)^{(n)}C_H^{(n)}}{2}
\left\{
|t-s|^{2H}-\sum_{j=1}^{(n-1)}(-1)^j\binom{2H}{j}
\left[ \left(\frac{t}{s}\right)^j|s|^{2H}+\left(\frac{s}{t}\right)^j|t|^{2H}\right]
\right\}, \label{eq:cov}
\end{eqnarray*}
where
$$
\binom{\alpha}{j} = \frac{\alpha(\alpha-1)\cdots(\alpha - (j-1))}{j!},
$$
and
$$
C_H^{(n)}
=
\frac{1}{\Gamma(2H+1)|\sin(\pi H)|}.
$$
For the variance we have
$$
\Var\left[B_H^{(n)}(t)\right] = C^{(n)}_H\binom{2H-1}{n-1}|t|^{2H}. 
$$

The following theorem provides an invertible Volterra representation on compact interval for the $n\mathrm{th}$ order fractional Brownian motion with respect to a standard Brownian motion generated from it.

\begin{thm}
\label{thm:nfbm_MG}
Let $n\geq 1$ be an integer and let $H\in(n-1,n)$. 
Let $W = (W(u), u\ge 0)$ be a one-sided standard Brownian motion.
Define a sequence of Volterra kernels $k_H^{(n)}$ recursively as
\begin{eqnarray}
k_H^{(1)}(t,u) &=& k_H(t,u), \label{eq:kernel_rec-1} \\ 
\label{eq:kernel_rec-2}
k_H^{(n)}(t,u) &=& \int_u^t k_{H-1}^{(n-1)}(s,u)\, \d s.
\end{eqnarray}
Then 
\begin{equation}\label{eq:mg-nfbm} 
B_H^{(n)}(t) = \int_0^t k_H^{(n)}(t,u)\, \d W(u), \quad t\ge 0,
\end{equation}
defines an $n$th order fractional Brownian motion. Moreover, the Brownian motion $W$ can be recovered from $B_H^{(n)}$ by 
\begin{equation}\label{eq:mg-nfbm-inv}
W(t)
=
\int_0^t k^{(-1)}_{H-n+1}(t,u)\, \d\, \frac{\d^{n-1}}{\d u^{n-1}}B_H^{(n)}(u).
\end{equation}
In particular, the filtrations of $W$ and $B_H^{(n)}$ coincide: $\F_W(t) = \F_{B_H^{(n)}}(t)$ for all $t\ge 0$.
\end{thm}

\begin{proof}
The proof is by induction. 

The case $n=1$ corresponds to the case of the fractional Brownian motion and follows from Proposition \ref{pro:fbm_MG}. 

Assume now that the claim is valid for some $k=n-1$. For $k=n$ we can, by square integrability of the kernels $k_H^{(n)}$, apply stochastic Fubini theorem to have
\begin{eqnarray*}
B^{(n)}_H(t) &=& \int_0^t k_H^{(n)}(t,u)\, \d W(u)\\
&=& \int_0^t\left[\int_u^t k_{H-1}^{(n-1)}(s,u)\, \d s\right]\,\d W(u)\\
&=& \int_0^t\left[\int_0^s k_{H-1}^{(n-1)}(s,u)\, \d W(u)\right]\,\d s\\
&=& \int_0^t B_{H-1}^{(n-1)}(s)\, \d s.
\end{eqnarray*}
Thus, \eqref{eq:mg-nfbm} defines the $n$th order fractional Brownian motion by \eqref{eq:diff-nfbm} together with the fact that $B_{H}^{(n)}(0)=0$. 
Furthermore, the claim \eqref{eq:mg-nfbm-inv} follows directly from \eqref{eq:mg-nfbm} together with \eqref{eq:diff-nfbm}. Finally, the equivalence of filtrations follows from Proposition \ref{pro:fbm_MG} together with the observation that as $B_{H}^{(n)}(0)=0$ for all $H$ and $n\geq 1$, the filtrations of $B_{H-1}^{(n-1)}$ and $B_H^{(n)}$ coincide by \eqref{eq:diff-nfbm}.
\end{proof}

\begin{rem}\label{rem:n-smooth}
As $n$ increases, so does the smoothness of the paths of $B_H^{(n)}$.  Since $W$ is not smooth, the representation \eqref{eq:mg-nfbm} implies that the kernels $k^{(n)}_H(t,u)$ have to become increasingly smooth in $t$ as $n$ increases.  This is also obvious from \eqref{eq:kernel_rec-1}--\eqref{eq:kernel_rec-2}.
\end{rem}

\begin{rem}
Let us consider writing the inversion formula \eqref{eq:mg-nfbm-inv} as
\begin{equation}\label{eq:mg-nfbm-inv-bis}
W(t) = \int_0^t k^{(-n)}_H(t,u)\, \d B_H^{(n)}(u).
\end{equation}
Since $B_H^{(n)}$ is smooth for $n>1$, and $W$ is not, the kernel $k^{(-n)}_H(t,u)$ in \eqref{eq:mg-nfbm-inv-bis} must be non-smooth in $t$.  Actually it is a Schwarz kernel, i.e., a proper distribution.  For example, if $n=2$, then we have 
$$
k_H^{(-2)}(t,u) = k^{(-1)}_{H-1}(t,u)\delta(t-u) - \frac{\p k_{H-1}^{(-1)}}{\p u}(t,u),
$$
where $\delta$ is the Dirac's delta function at point $0$, and the partial derivative $\frac{\p k_{H-1}^{(-1)}}{\p u}(t,u)$ has to be understood in the sense of distributions.
\end{rem}


With the help of Theorem \ref{thm:nfbm_MG} we can provide the transfer principle for the $n$th order fractional Brownian motion. In what follows, the Wiener integral with respect to the $n$th order fractional Brownian motion is defined in the spirit of Definition \ref{dfn:isonormal} as follows.  Define a Hilbert space $\H_H^{(n)}=\H_H^{(n)}([0,T])$ such that:
\begin{enumerate}
\item indicators $\1_t := \1_{[0,t)}$, $t\le T$, belong to $\H_H^{(n)}$.
\item $\H_H^{(n)}$ is endowed with the inner product $\la \1_t,\1_s\ra_{\H_H^{(n)}} := r_H^{(n)}(t,s)$,
where the covariance $r_H^{(n)}$ is given by \eqref{eq:cov}.
\end{enumerate}
Then $B_H^{(n)}= ( B_H^{(n)}(t), t\in [0,T])$ is the isonormal process associated with the Hilbert space $\H_H^{(n)} = \H_H^{(n)}([0,T])$ and 
the Wiener integral 
$$
\int_0^T f(s) \,\d B^{(n)}_H(s) = B_H^{(n)}(f)
$$
of $f \in \H_H^{(n)}$ with respect to the $n$th order fractional Brownian motion $B_H^{(n)}$ is a centred Gaussian random variable, and for $f,g \in \H_H^{(n)}$ we have
$$
\E \left[\int_0^T f(s) \,\d B^{(n)}_H(s)\int_0^T g(s) \,\d B^{(n)}_H(s)\right] = \la f,g\ra_{\H_H^{(n)}}.
$$

The following provides a transfer principle for the $n$th order fractional Brownian motion, in the spirit of Theorem \ref{thm:fbm_transfer}.

\begin{thm}[Transfer principle]\label{thm:nfbm_transfer}
Let $B_H^{(n)}$ be the $n$th order fractional Brownian motion on $[0,T]$ with $n\geq 1$ and $H\in(n-1,n)$. Define an operator 
$
k^{(n)*} : \H_H^{(n)} \mapsto L^2([0,T])
$ by linearly extending
$$
k_H^{(n)*} \1_t = k^{(n)}_H(t,\cdot).
$$
Then $k_H^{(n)*}$ provides an isometry between $\H_H^{(n)}$ and $L^2([0,T])$. Moreover, for any $f \in \H_H^{(n)}$ we have
$$
\int_0^T f(u)\, \d B^{(n)}_H(u) = \int_0^T \left(k^{(n)*}_H f\right) (u)\, \d W(u).
$$
Furthermore, if $n\geq 2$, then $L^2([0,T]) \subset \H_H^{(n)}$, and for any $f \in L^2([0,T])$ we have
\begin{equation}
\label{eq:nfbm_L2}
\left(k^{(n)*}_H f\right)(u) = \int_0^T f(t)k^{(n-1)}_{H-1}(t,u)\,\d t.
\end{equation}
\end{thm}

\begin{proof}
The first part follows from the similar arguments as in the general case (see \cite{fredholm}). For the reader's convenience we present the main arguments. 

Assume first that $f$ is an elementary function of form
$$
f(t) = \sum_{k=1}^{(n)} a_k \textbf{1}_{A_k}
$$
for some disjoint intervals $A_k=(t_{k-1},t_k]$. Then the claim follows by the very definition of the operator $k_H^{(n)*}$ and Wiener integral with respect to $B_H^{(n)}$ together with representation \eqref{eq:mg-nfbm}, and this shows that $k_H^{(n)*}$ provides an isometry between $\H_H^{(n)}$ and $L^2([0,T])$. Hence $\H_H^{(n)}$ can be viewed as a closure of elementary functions with respect to $\| f\|_{\H_H^{(n)}} = \| k_H^{(n)*}f\|_{L^2([0,T])}$ which proves the claim. 

In order to complete the proof, we need to verify \eqref{eq:nfbm_L2}. For this, assume again that $f$ is an elementary function. Then it is straightforward to check that (see also \cite[Example 4.1]{fredholm})
$$
\left(k^{(n)*}_H f\right)(u) = \int_0^T f(t)\, \frac{\partial k^{(n)}_{H}(t,u)}{\partial t}\, \d t,
$$
and thus, thanks to \eqref{eq:kernel_rec-1}--\eqref{eq:kernel_rec-2}, we have \eqref{eq:nfbm_L2} for all elementary $f$. Let now $f\in L^2([0,T])$
and take a sequence $f_j$ of elementary functions such that $\Vert f_j - f\Vert_{L^2([0,T])}$. Then $(f_j, j\in\mathbb{N})$ 
is a Cauchy sequence in $L^2([0,T])$. We have
\begin{equation*}
\left(k^{(n)*}_H f_j\right)(u) = \int_0^T f_j(t)k^{(n-1)}_{H-1}(t,u)\,\d t,
\end{equation*}
and thus
\begin{eqnarray*}
\Vert f_j\Vert_{\H_H^{(n)}}^2 &=& \Vert k^{(n)*}_H f_j\Vert_{L^2}^2\\
&=& \int_0^T \left(\int_0^T f_j(t)k^{(n-1)}_{H-1}(t,u)\,\d t\right)^2 \,\d u\\
&=& \int_0^T \int_0^T \left(k^{(n-1)}_{H-1}(t,u)\right)^2 \,\d t\,\d u \Vert f_j\Vert_{L^2([0,T])}^2\\
&\leq & C_{T,H,n}\Vert f_j\Vert_{L^2([0,T])}^2.
\end{eqnarray*}
The same argument applied to $f_j - f_i$ shows that $(f_j,j\in\mathbb{N})$ is also a Cauchy sequence in $\H_H^{(n)}$. Thus $L^2([0,T]) \subset \H_H^{(n)}$ and \eqref{eq:nfbm_L2} holds.
\end{proof}

\begin{rem}
We stress that for $n=1$ the equation \eqref{eq:nfbm_L2} does not hold. On the other hand, in this case we also have $L^2([0,T]) \subset \H^{(1)}_H$ for $H>\frac12$, while for $H<\frac12$ we have $\H^{(1)}_H \subset L^2([0,T])$. Finally, we also remark that, by the characterisation of the space $\H^{(1)}_H$ in \cite{Pipiras-Taqqu-2001} one can characterise the space $\H^{(1)}_H$ simply by using Fubini's theorem and \eqref{eq:kernel_rec-1}--\eqref{eq:kernel_rec-2}.
\end{rem}
\begin{rem}
The transfer principle provided in Theorem \ref{thm:nfbm_transfer} extends in a straightforward manner to multiple Wiener integrals. Moreover, the transfer principle can be used as a simple approach to stochastic analysis and Malliavin calculus with respect to the $n$th order fractional Brownian motion. For the details on the topic, we refer to \cite{fredholm}.
\end{rem}

\section{Applications}
\label{sec:applications}
\subsection{Equivalence in law}
We first investigate the equivalence of law problem.  For the treatment of the problem in the classical fractional Brownian motion or sheet case, see \cite{Sottinen-2004,Sottinen-Tudor-2006}.

Gaussian processes are either singular or equivalent in law.  By Hitsuda's representation theorem \cite{Hitsuda-1968} a Gaussian process $\tilde W =(\tilde W(t), t\in [0,T])$, is equivalent in law to a Brownian motion $W=(W(t),t\in [0,T])$ if and only if there there exists a Volterra kernel $b\in L^2([0,T]^2)$ and a function $a\in L^2([0,T])$ such that  
\begin{equation}\label{eq:equiv-bm}
\tilde W(t) =  
W(t)  - \int_0^t\int_0^s b(s,u) \d W(u) \d s + \int_0^t a(s)\, \d s.
\end{equation}
Here the Brownian motion $W$ is constructed from $\tilde W$ as
$$
W(t) = 
\tilde W(t) - \int_0^t \int_0^s b^*(s,u)( \d\tilde W(u) - a(u)\d u) \d s,
- \int_0^t a(s)\, \d s.
$$
where $b^*$ is the unique resolvent Volterra kernel of $b$ solving the equation
$$
\int_s^t b^*(t,u)b(u,s)\d u
=
b^*(t,s) + b(t,s)
=
\int_s^t b(t,u)b^*(u,s)\d u.
$$ 
The resolvent kernel can be constructed by using Neumann series, see \cite{Smithies-1958} for details.

The log-likelihood ratio of model $\tilde W$ over $W$ is
\begin{eqnarray}
\ell(t) 
&=& \log \frac{\d\tilde\P}{\d\P}\Big|\F_t  \nonumber \\
&=&
\int_0^t\left[\int_0^s b(s,u)\, \d W(u) + a(s)\right]\d W(s)  \nonumber \\
& &  -\frac12\int_0^t \left[\int_0^sb(s,u)\, \d W(u) + a(s)\right]^2\, \d s.
\label{eq:lr}
\end{eqnarray}

Consider then a Gaussian process $\tilde B_H^{(n)} = (\tilde B_H^{(n)}(t), t\in [0,T])$.  This process is equivalent to the $n$th order fractional Brownian motion if and only if the process
\begin{equation}\label{eq:mg-nfbm-inv-tilde}
\tilde W(t)
=
\int_0^t k^{(-1)}_{H-n+1}(t,u)\, \d\, \frac{\d^{n-1}}{\d u^{n-1}}\tilde B_H^{(n)}(u).
\end{equation}
is equivalent to a Brownian motion.  From \eqref{eq:equiv-bm} and \eqref{eq:mg-nfbm} it follows that
\begin{eqnarray}
\lefteqn{\tilde B_H^{(n)}(t)} \nonumber \\ &=& 
\int_0^t k_H^{(n)}(t,s)\d\tilde W(s) \nonumber \\
&=&
B_H^{(n)}(t) - \int_0^t k_H^{(n)}(t,s)\int_0^s b(s,u)\,\d W(u)\, \d s
+ \int_0^t k_H^{(n)}(t,s) a(s)\, \d s. \label{eq:tilde-nfbm}
\end{eqnarray}

Let us collect the discussion above as a theorem:

\begin{thm}
A Gaussian process $\tilde B_H^{(n)}$ is equivalent in law to an $n$th order fractional Brownian motion on $[0,T]$ if and only if there exists a Volterra kernel $b\in L^2([0,T]^2)$ and a function $a\in L^2([0,T])$ such that $\tilde B_H^{(n)}$ admits the representation \eqref{eq:tilde-nfbm}, where $\tilde W$ is constructed from $\tilde B_H^{(n)}$ by \eqref{eq:mg-nfbm-inv-tilde} and $W$ is a Brownian motion connected to $\tilde W$ via \eqref{eq:equiv-bm}.  The log-likelihood ratio of $\tilde B_H^{(n)}$ over $B_H^{(n)}$ is given by \eqref{eq:lr}.
\end{thm}

\begin{proof}
By \eqref{eq:mg-nfbm-inv}, the process $\tilde B_H^{(n)}$ is equivalent to the $n$th order fractional Brownian motion if and only if the process 
\begin{equation}
\tilde W(t)
=
\int_0^t k^{(-1)}_{H-n+1}(t,u)\, \d\, \frac{\d^{n-1}}{\d u^{n-1}}\tilde B_H^{(n)}(u).
\end{equation}
is equivalent to a Brownian motion.  From \eqref{eq:equiv-bm} and \eqref{eq:mg-nfbm} it follows that
\begin{eqnarray}
\lefteqn{\tilde B_H^{(n)}(t)} \nonumber \\ &=& 
\int_0^t k_H^{(n)}(t,s)\d\tilde W(s) \nonumber \\
&=&
B_H^{(n)}(t) - \int_0^t k_H^{(n)}(t,s)\int_0^s b(s,u)\,\d W(u)\, \d s
+ \int_0^t k_H^{(n)}(t,s) a(s)\, \d s
\end{eqnarray}
which concludes the proof.
\end{proof}

\begin{rem}\label{rem:drift}
For transformations $\tilde B_H^{(n)}(t) =  B_H^{(n)}(t) + A(t)$ that are equivalent in law to $B_H^{(n)}$ on $[0,T]$ we have
$$
A(t) = \int_0^t k_H^{(n)}(t,s)a(s)\, \d s
$$
for some $a\in L^2([0,T])$.
Consequently, $A$ is $H+\frac12$ times fractionally differentiable with $\frac{\d^j}{\d t^j}A(0)=0$ for all $j<n$.  Otherwise, in principle, the function $A$ can be filtered out with probability one given continuous data on any interval $[0,T]$ with $T>0$.  In particular, it follows that the drift $\alpha$ in signal $X(t)= B_H^{(n)}(t)+ \alpha t$ for $n\ge 2$, can be completely determined from continuous observations on any interval $[0,T]$. Indeed, $\alpha = \frac{\d}{\d t}X(0)$.
\end{rem}

\subsection{Prediction of the $n$the order fractional Brownian motion}
The equivalence of filtrations generated by the Brownian motion and the $n$th order fractional Brownian motion is of uttermost important in prediction. Indeed, this guarantees that the prediction of the $n$th order fractional Brownian motion can be done by using the same approach as taken in \cite{Sottinen-Viitasaari-2017b}.

\begin{thm}
The regular conditional distribution of the $n$th order fractional Brownian motion $B_H^{(n)}=(B_H^{(n)}(t), t\in [0,T])$ conditioned on the information $\F_{B_H^{(n)}}(u) = \sigma\{B_H^{(n)}(v), v\le u\}$, is Gaussian process with random mean $\hat B_H^{(n)}(\cdot|u)$ given by 
\begin{equation}\label{eq:pred-mean}
\hat B_H^{(n)}(t|u) = B_H^{(n)}(u) + \int_0^u \left[ k_H^{(n)}(t,v) - k_H^{(n)}(u,v)\right]\, \d W(v)
\end{equation}
and a deterministic covariance $\hat r_H^{(n)}(\cdot,\cdot|u)$ given by
\begin{equation}\label{eq:pred-cov}
\hat r_H^{(n)}(t,s|u) 
=
r_H^{(n)}(t,s) - \int_{0}^{u} k_H^{(n)}(t,v)k_H^{(n)}(s,v)\, \d v.
\end{equation}
\end{thm}
\begin{proof}
By the Gaussian correlation theorem (see Janson \cite{Janson-1997}) the conditional law is Gaussian with mean 
$$
t\mapsto\E\left[ B^{(n)}_H(t)\,\Big|\, \F_{B_H^{(n)}}(u)\right]
$$
and a covariance function
$$
(t,s)\mapsto\hat r_H^{(n)}(t,s|u) = \E\left[\left(B_H^{(n)}(t)-\hat B_H^{(n)}(t|u)\right)\left(B_H^{(n)}(s)-\hat B_H^{(n)}(s|u)\right)
\,\Big|\, \F_{B_H^{(n)}}(u)\right].
$$

We start by proving equation \eqref{eq:pred-mean}. Since the $n$th order fractional Brownian motion $B^{(n)}_H$ admits the representation \eqref{eq:mg-nfbm} and the filtrations of $B^{(n)}_H$ and the Brownian motion $W$ coincide, the prediction mean of $B_H^{(n)}$ given observations $\F_{B_H^{(n)}}(u)$ is the same as the prediction mean under the observations $\F_{W}(u)$:
\begin{eqnarray*}
\hat B_H^{(n)}(t|s) 
&=&
\E\left[ B^{(n)}_H(t)\,\Big|\, \F_{B_H^{(n)}}(u)\right] \\
&=&
\E\left[ B^{(n)}_H(t)\,\Big|\, \F_{W}(u)\right]
\end{eqnarray*}
By using \eqref{eq:mg-nfbm} and the independence of Brownian increments we obtain
\begin{eqnarray*}
\hat B_H^{(n)}(t|u) 
&=&
\E\left[ \int_0^t k_H^{(n)}(t,v)\, \d W(v)\,\Big|\, \F_{W}(u)\right] \\
&=&
\E\left[ \int_0^u k_H^{(n)}(t,v)\, \d W(v)
+ \int_u^t k_H^{(n)}(t,v)\, \d W(v)
\,\Big|\, \F_{W}(u)\right] \\
&=&
\E\left[ \int_0^u k_H^{(n)}(t,v)\, \d W(v)
\,\Big|\, \F_{W}(u)\right] + \E\left[ \int_u^t k_H^{(n)}(t,v)\, \d W(v)
\right] \\
&=&
\int_0^u k_H^{(n)}(t,v)\, \d W(v).
\end{eqnarray*}
This formula can also be written as 
\begin{eqnarray*}
\hat B_H^{(n)}(t|u) 
&=&
\int_0^u k_H^{(n)}(t,v)\, \d W(v) \\
&=&
\int_0^u k_H^{(n)}(u,v) \, \d W(v) 
+  \int_0^u \left[ k_H^{(n)}(t,v) - k_H^{(n)}(u,v)\right]\, \d W(v) \\
&=&
B_H^{(n)}(u) + \int_0^u \left[ k_H^{(n)}(t,v) - k_H^{(n)}(u,v)\right]\, \d W(v),
\end{eqnarray*}
proving \eqref{eq:pred-mean}. 

It remains to prove \eqref{eq:pred-cov}. Proceeding similarly, the conditional covariance can be calculated as 
\begin{eqnarray*}
\lefteqn{\hat r_H^{(n)}(t,s|u)} \\ 
&=&
\E\left[\left(B_H^{(n)}(t)-\hat B_H^{(n)}(t|u)\right)\left(B_H^{(n)}(s)-\hat B_H^{(n)}(s|u)\right)
\,\Big|\, \F_{B_H^{(n)}}(u)\right] \\
&=&
\E\left[\int_u^t k_H^{(n)}(t,v)\, \d W(v)\int_u^s k_H^{(n)}(s,v)\, \d W(v)
\,\Big|\, \F_{W}(u)\right] \\
&=&
\int_{u}^{\min(t,s)} k_H^{(n)}(t,v)k_H^{(n)}(t,v)\, \d v.
\end{eqnarray*}
Since
$$
r_H^{(n)}(t,s) = \int_0^{\min(t,s)} k_H^{(n)}(t,u) k_H^{(n)}(s,u)\, \d u,
$$
formula \eqref{eq:pred-cov} follows from this.  This concludes the proof.
\end{proof}

\begin{rem}
By using a the transfer principle, it is possible to write \eqref{eq:pred-mean} as
$$
\hat B_H^{(n)}(t|u) = B_H^{(n)}(u) + \int_0^u \Psi_H^{(n)}(t,u,v)\, \d B_H^{(n)}(v),
$$
where $\Psi_H^{(n)}(t,u,\cdot)$ is a Schwarz kernel.
\end{rem}

\begin{exa}
The prediction mean for a transformation $f(B_H^{(n)}(t))$ given the observation $\F_{B_H^{(n)}}(u)$ can be calculated as
\begin{eqnarray}
\lefteqn{\E\left[f(B_H^{(n)}(t))\,\big|\, \F_{B_H^{(n)}}(u)\right]} \nonumber \\
&=&
\frac{1}{\sqrt{2\pi}\hat r_H^{(n)}(t,t|u)}\int_{-\infty}^{\infty}
f(x)\exp\left\{-\frac12\frac{\left(x-\hat B_n^H(t|u)\right)^2}{\hat r_H^{(n)}(t,t|u)}\right\}
\, \d x. \label{eq:exa-pred}
\end{eqnarray}
(If the transformation $f$ in injective, then \eqref{eq:exa-pred} is also the prediction under the observations $\sigma\{f(B_H^{(n)}(v)),v \le u\}$.)
More generally, for 
$f(B_H^{(n)}(\t)) = f(B_H^{(n)}(t_1),\ldots,B_H^{(n)}(t_k))$
we have
\begin{eqnarray}
\lefteqn{\E\left[f(B_H^{(n)}(\t))\,\big|\, \F_{B_H^{(n)}}(u)\right]} \nonumber \\
&=&
(2\pi)^{-k/2}\left|\hat r_H^{(n)}(\t|u)\right|^{-1/2}\int_{\R^k} 
f(\mathbf{x}) \nonumber \\
& & 
\exp\left\{-\frac12(\mathbf{x}-\hat B_n^H(\t|u))'\hat r_H^{(n)}(\t|u)^{(-1)}(\mathbf{x}-\hat B_n^H(\t|u))\right\}
\, \d\mathbf{x}, \label{eq:exa-pred2}
\end{eqnarray}
where
\begin{eqnarray*}
\hat r_H^{(n)}(\t|u) &=& \big[\hat r_H^{(n)}(t_i,t_j|u)\big]_{i,j=1}^k, \\
\hat B_H^{(n)}(\t) &=& \big[\hat B_H^{(n)}(t_i)\big]_{i=1}^k.
\end{eqnarray*}
\end{exa}
\begin{rem}
The no-information and full-information asymptotics of the conditional covariance \eqref{eq:pred-cov} can be computed similarly as in \cite[Proposition 3.2 and Proposition 3.3]{Sottinen-Viitasaari-2017b}. Actually, one can simply use 
the result for the standard fractional Brownian motion, \cite[Proposition 3.2 and Proposition 3.3]{Sottinen-Viitasaari-2017b} and then apply \eqref{eq:kernel_rec-1}--\eqref{eq:kernel_rec-2} together with the induction to obtain asymptotic expansions. Indeed, this follows directly from the Fubini's theorem. We leave the details to the reader.
\end{rem}

\section{Simulations}\label{sec:simu}

In this section we illustrate the smoothening effect by simulating the paths of $n$th order fractional Brownian motion for different values of the Hurst index $H$ and for values of $n=1,2,3,4$. In all of the simulations we have used time interval $[0,1]$ with $2^{12} = 4096$ grid points, and the path of the fractional Brownian motion is simulated by using the fast Fourier transform. In figures 1--5 we have generated one sample path of the fractional Brownian motion $B_H$ with different values of $H$, and the corresponding sample paths of $B_H^{(2)}$, $B_H^{(3)}$, and $B_H^{(4)}$, by using \eqref{eq:diff-nfbm}. 

Overall, the smoothening effect is clearly visible from the pictures for all values $H \in \{0.1,0.25,0.5,0.75,0.9\}$. For all different values of $H$, the paths of the processes $B_H^{(3)}$ and $B_H^{(4)}$ look rather smooth. This is clear also from the theoretical point of view, as $B_H^{(3)}$ is already twice continuously differentiable. The differences of the value $H$ can be seen obviously from the paths of $B_H$ itself, but also from the paths of $B_H^{(2)}$. Indeed, the paths of $B_H^{(2)}$ corresponding to values $H=0.1$ and $H=0.25$ in figures 1 and 2 are not clearly smooth. From theoretical point of view, the paths are ''barely'' continuously differentiable, as the derivative is very rough. In contrast, the paths of $B_H^{(2)}$ corresponding to values $H=0.75$ and $H=0.9$ are ''almost'' twice continuously differentiable, as the path of the $B_H$ itself is almost differentiable in the sense that the H\"older index is close to 1 corresponding to the smooth case. This can be seen also from figures 4 and 5. Indeed, comparing $B_{0.1}^{(3)}$ in Figure 1(C) to $B_{0.9}^{(2)}$ in Figure 5(B) the smoothness looks very similar. This is not surprising from theoretical point of view, as the regularity index of $B_{0.1}^{(3)}$ is 2.1 (meaning $B_{0.1}^{(3)}$ is twice continuously differentiable and the second derivative is 0.1-H\"older) while the regularity index of $B_{0.9}^{(2)}$ is 1.9. 

We also stress that in all of the Figures 1--5 it seems that the values seem to get smaller as $n$ increases. This phenomena is again supported by the theory. Indeed, first of all we are integrating over $[0,t]$ for $t\leq 1$ which means that the extrema points gets smaller in absolute value. In addition, the small deviation probabilities gets higher (see \cite{aurzada-2011,aurzada-2013,aurzada-2009}) as $n$ increases. That is, for larger value of $n$ the process is smoother which implies that the process stays in a small $\epsilon$-ball around starting point $B_H^{(n)}(0)=0$ with higher probability. 
\begin{figure}[!htbp]
  \centering
  \begin{tabular}{cc}
  \subfloat[1. order]{\includegraphics[width=0.5\textwidth]{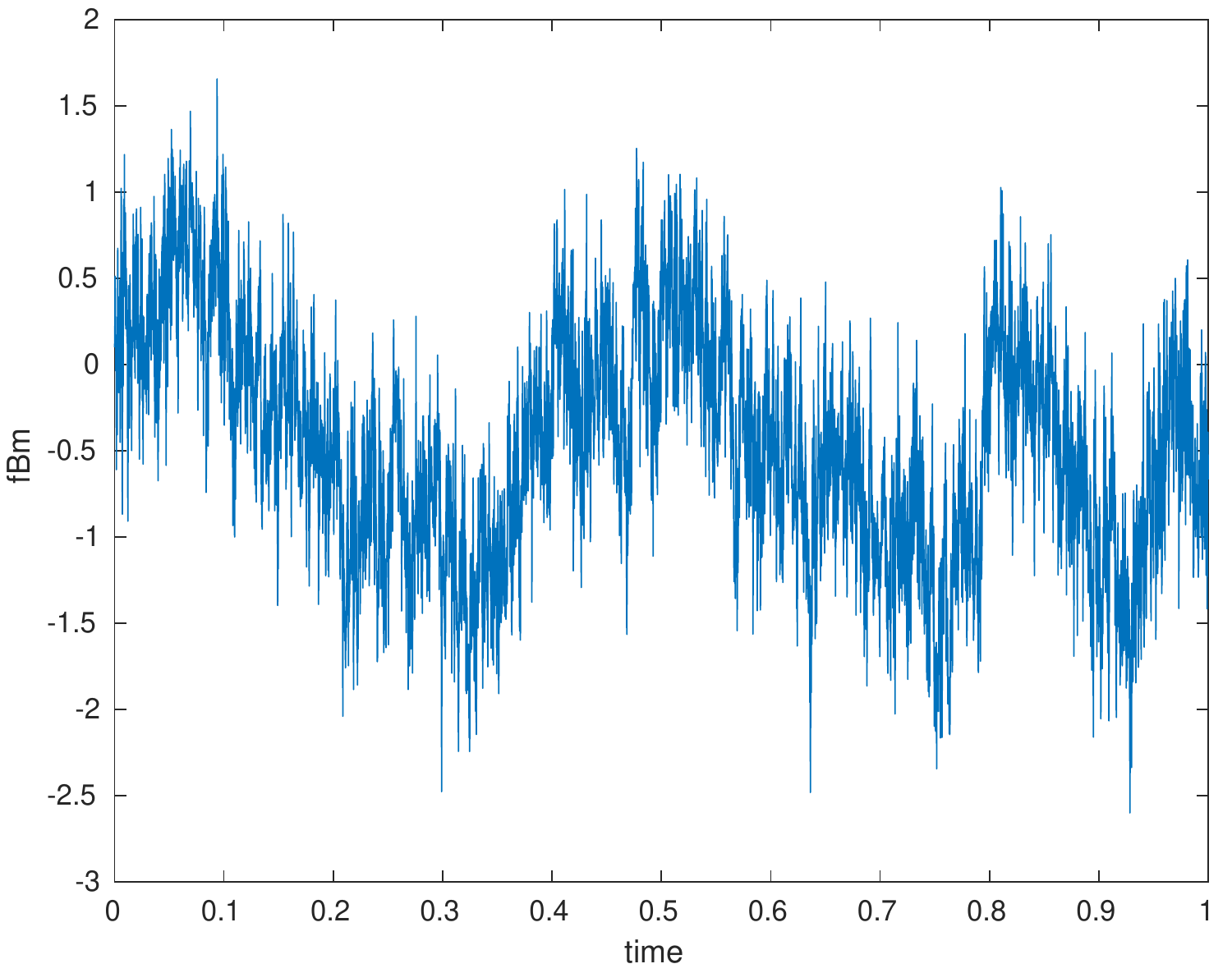}\label{fig:fbm1}}
  \subfloat[2. order]{\includegraphics[width=0.5\textwidth]{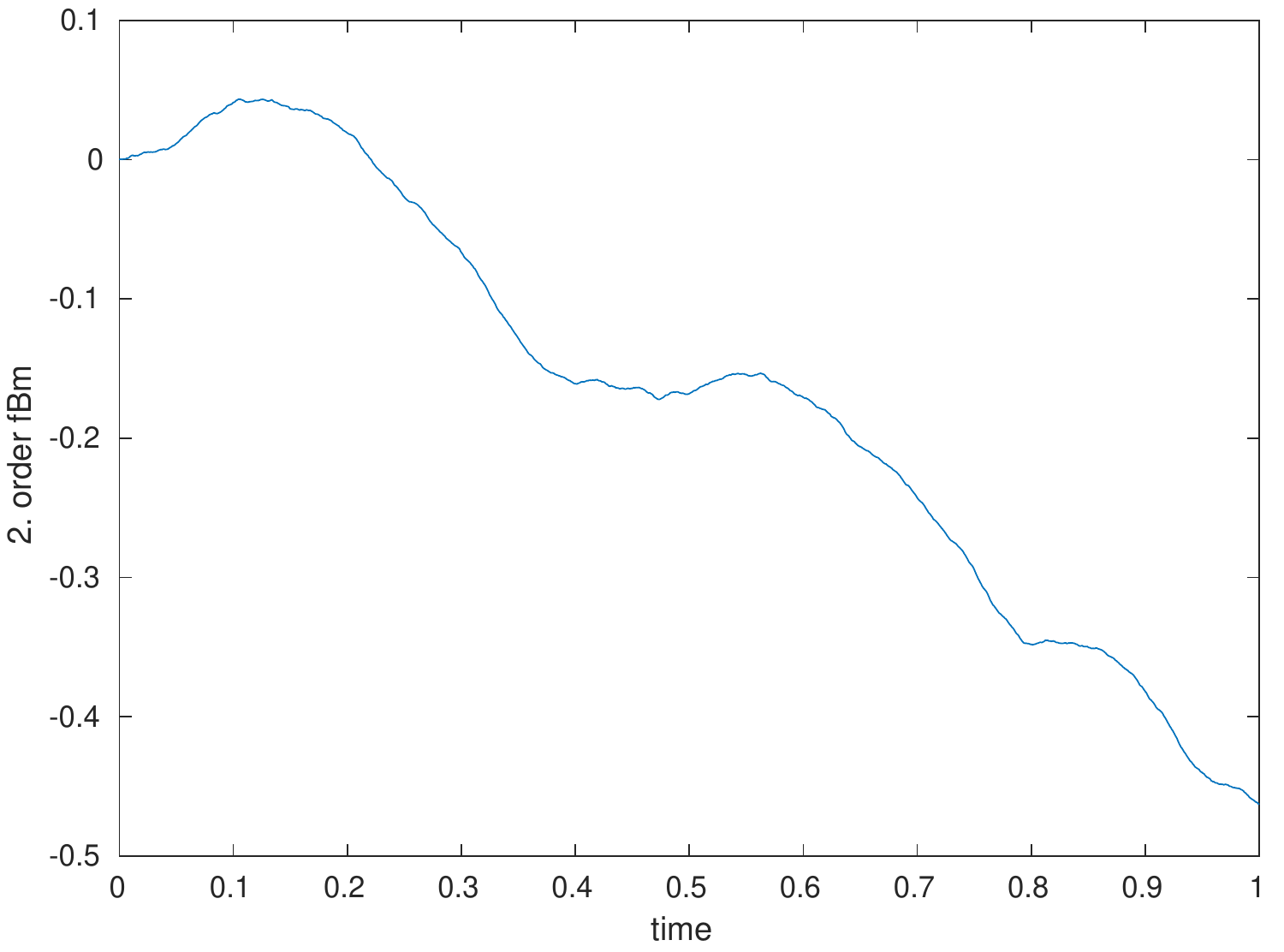}\label{fig:2fbm1}}\\ 
    \subfloat[3. order]{\includegraphics[width=0.5\textwidth]{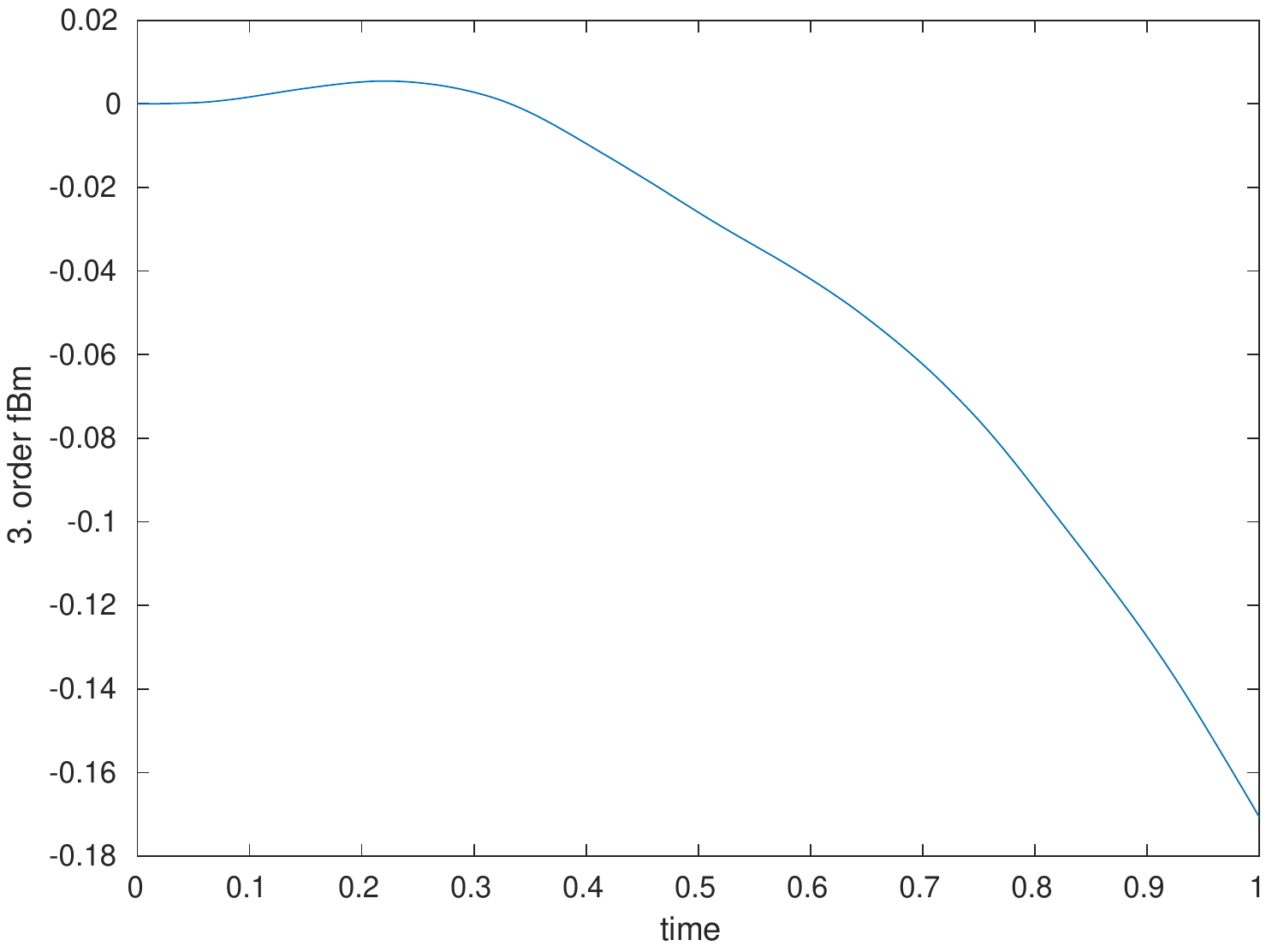}\label{fig:3fbm1}}
  \subfloat[4. order.]{\includegraphics[width=0.5\textwidth]{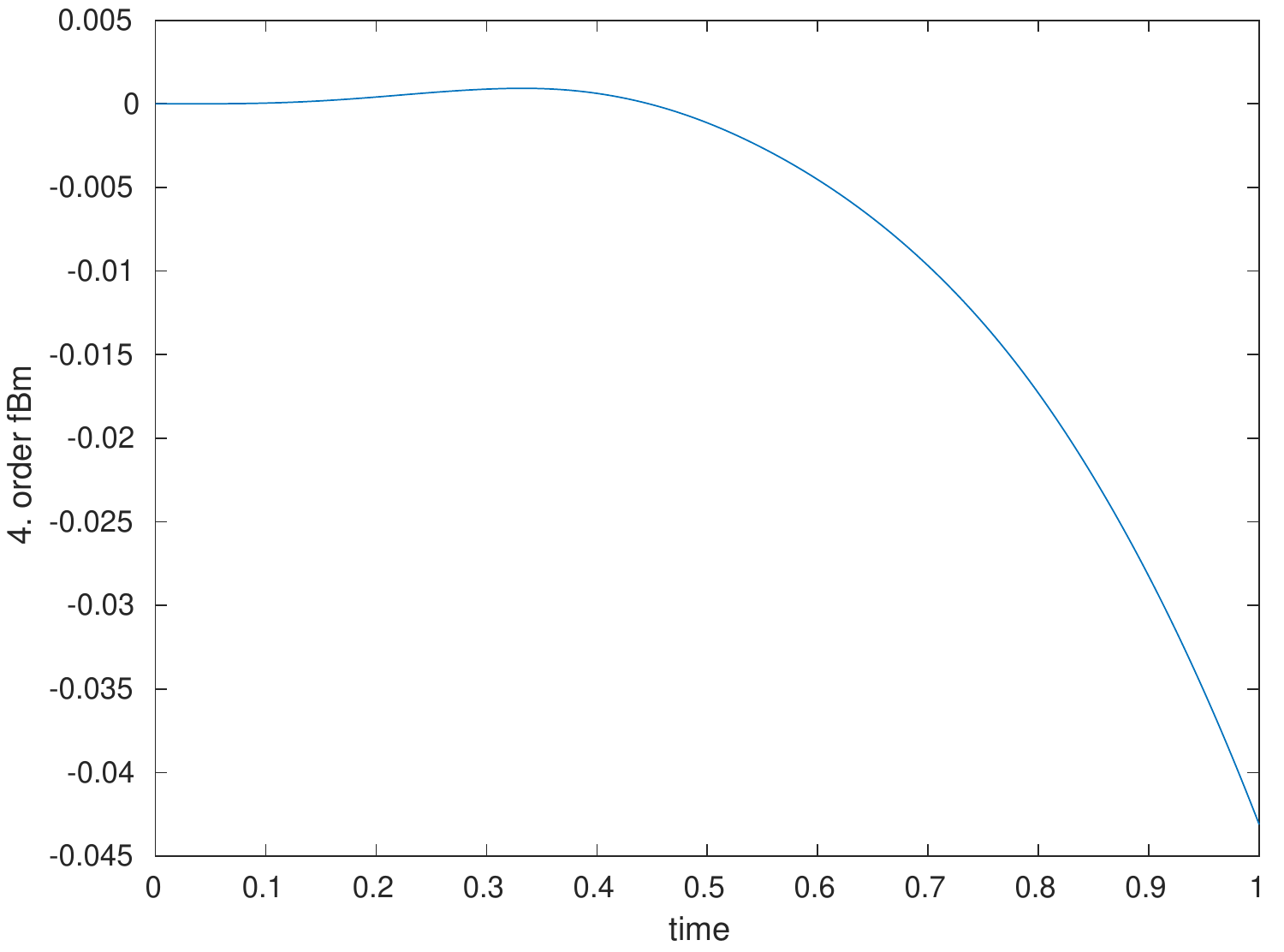}\label{fig:4fbm1}}\\
\end{tabular}
  \caption{1-4. order fractional Brownian motion with $H=0.1$.}
\end{figure}

\begin{figure}[!htbp]
  \centering
  \begin{tabular}{cc}
  \subfloat[1. order]{\includegraphics[width=0.5\textwidth]{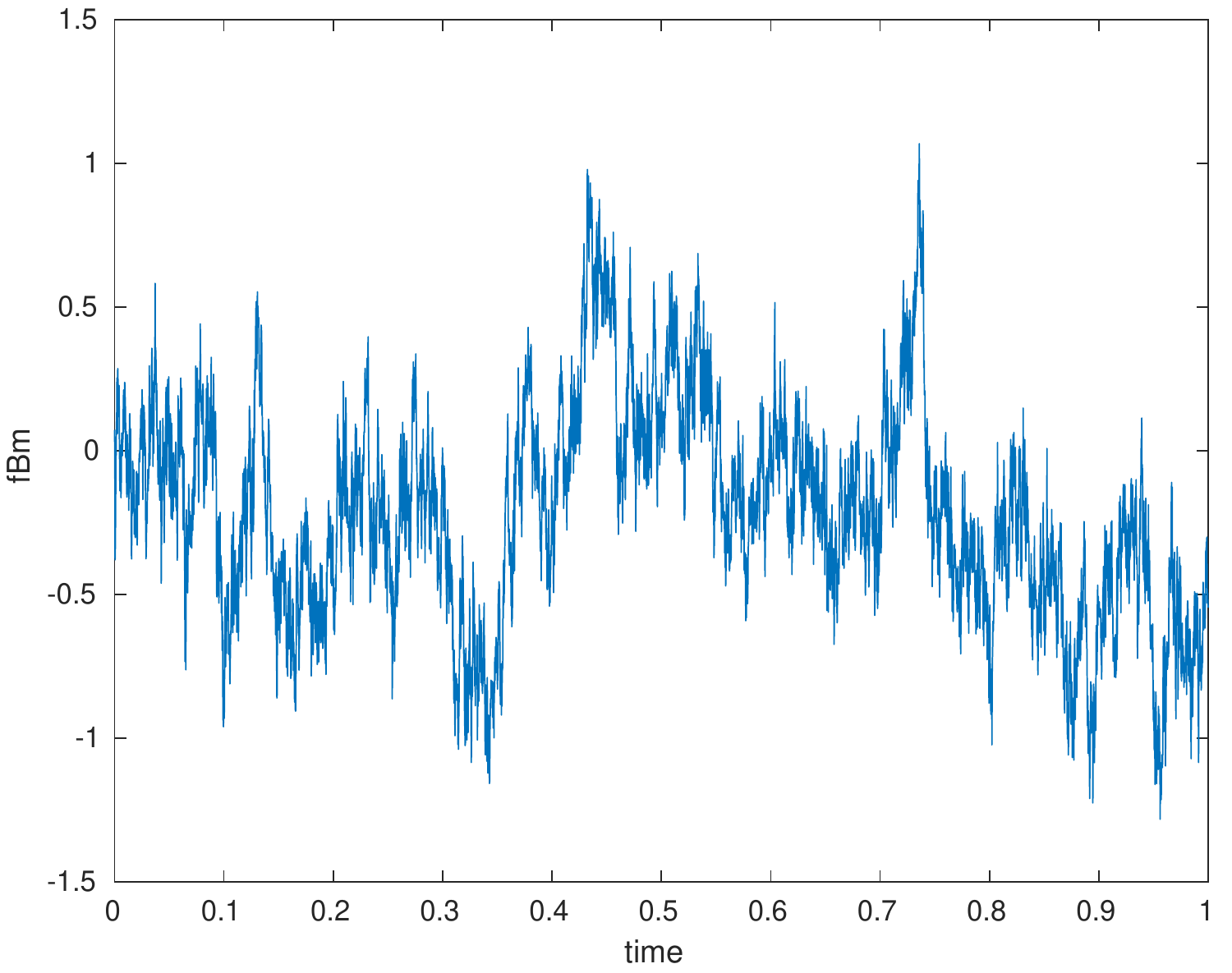}\label{fig:fbm25}}
  \subfloat[2. order]{\includegraphics[width=0.5\textwidth]{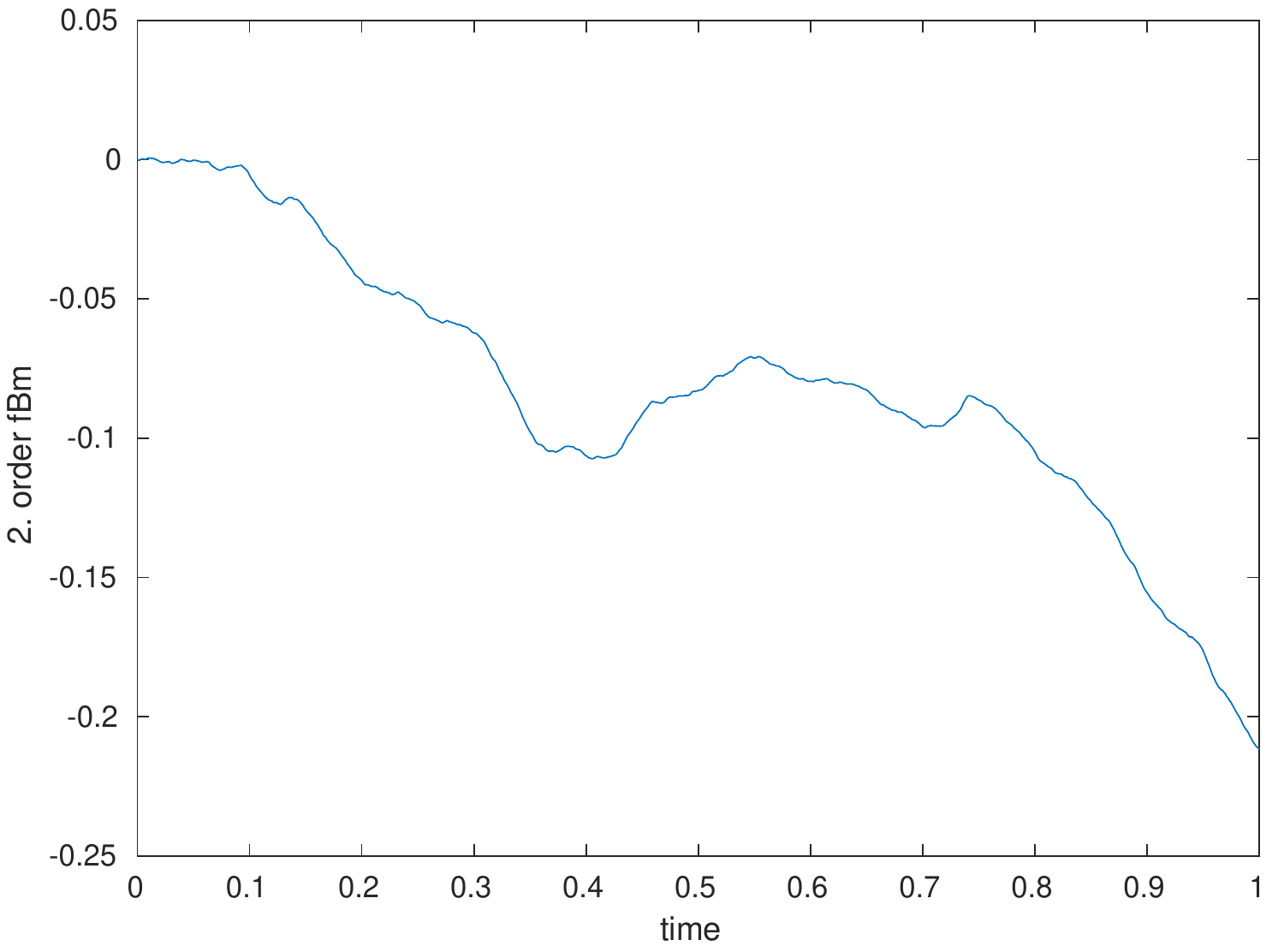}\label{fig:2fbm25}}\\ 
    \subfloat[3. order]{\includegraphics[width=0.5\textwidth]{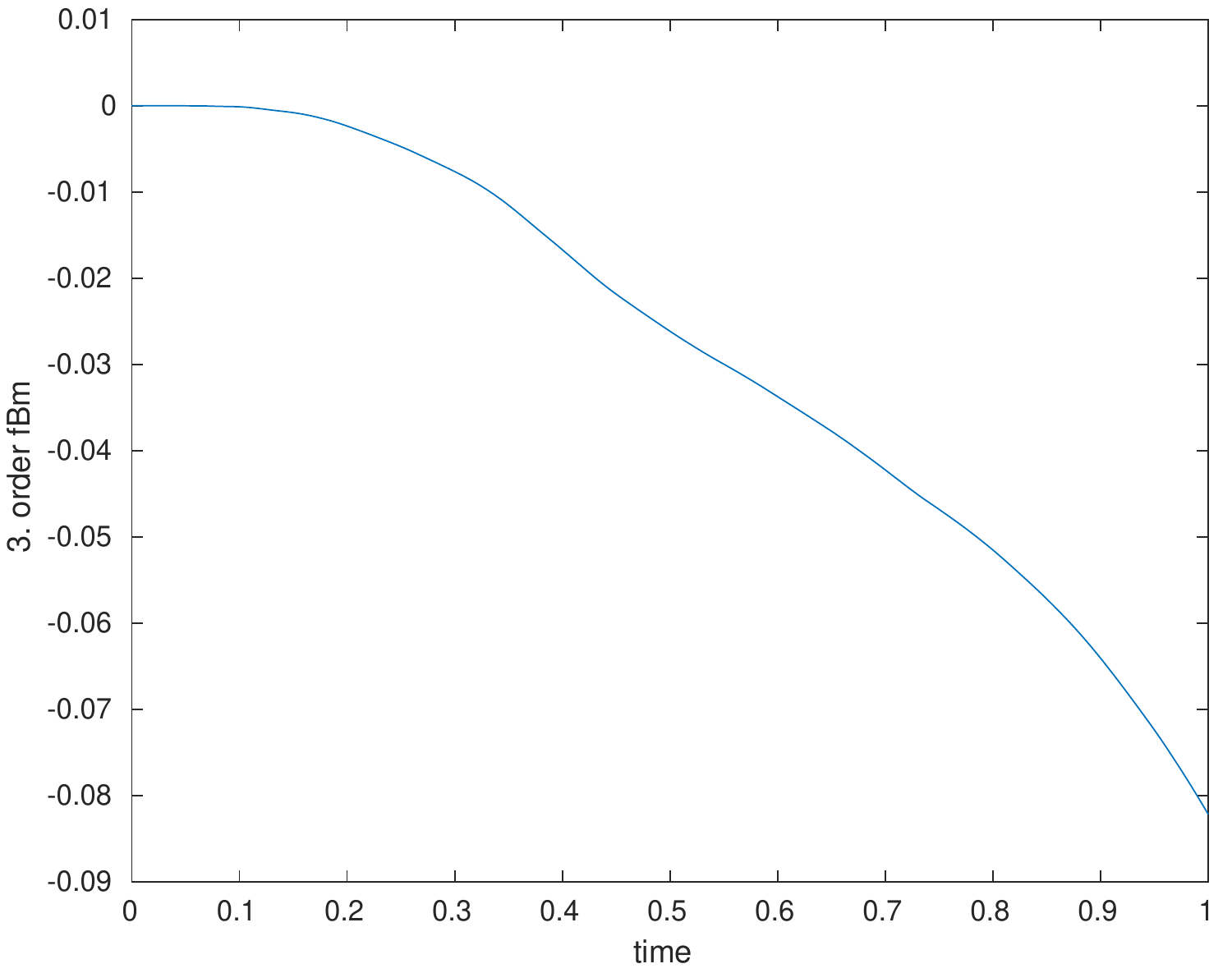}\label{fig:3fbm25}}
  \subfloat[4. order.]{\includegraphics[width=0.5\textwidth]{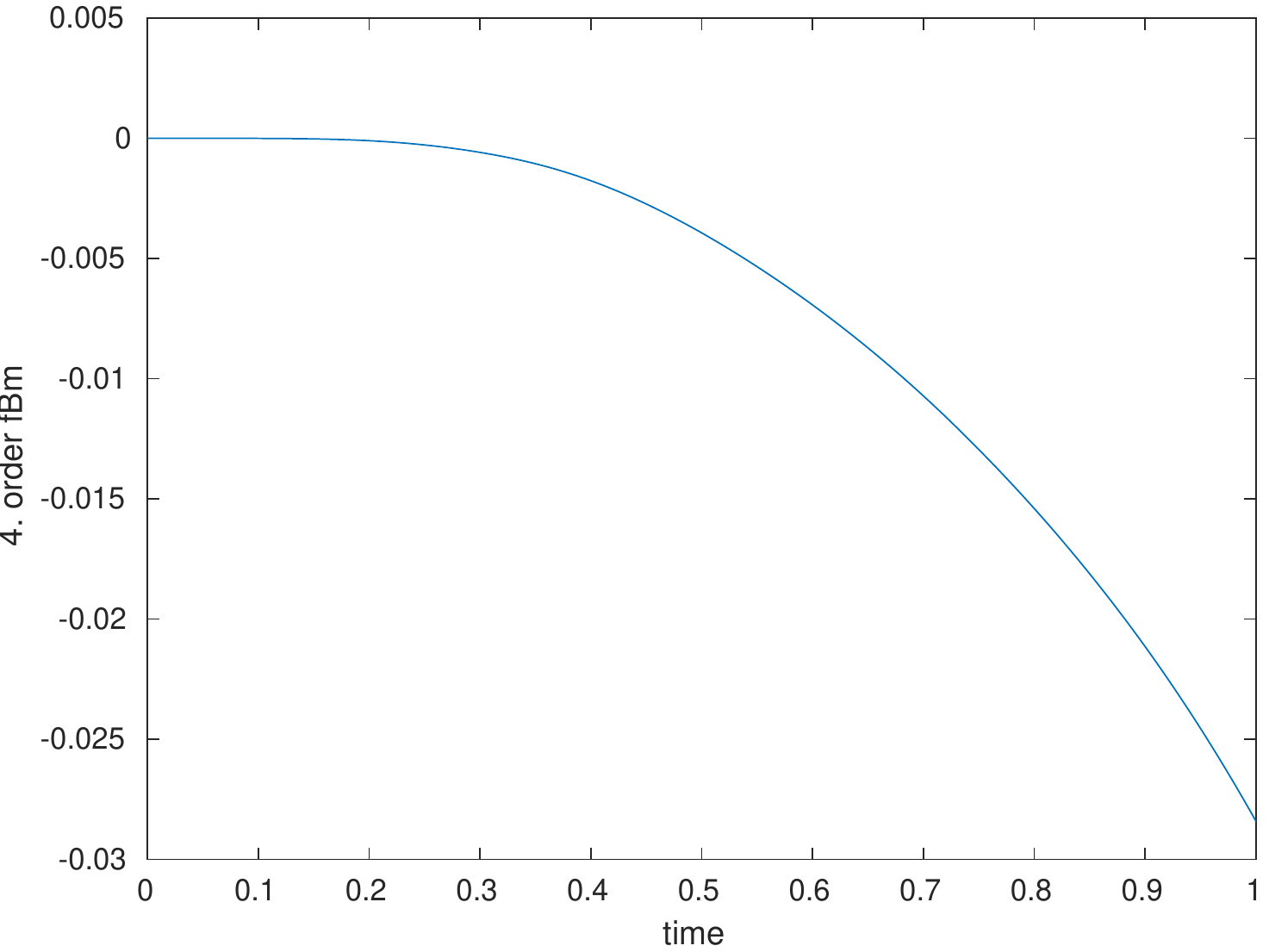}\label{fig:4fbm25}}\\
\end{tabular}
  \caption{1-4. order fractional Brownian motion with $H=0.25$.}
\end{figure}

\begin{figure}[!htbp]
  \centering
  \begin{tabular}{cc}
  \subfloat[1. order]{\includegraphics[width=0.5\textwidth]{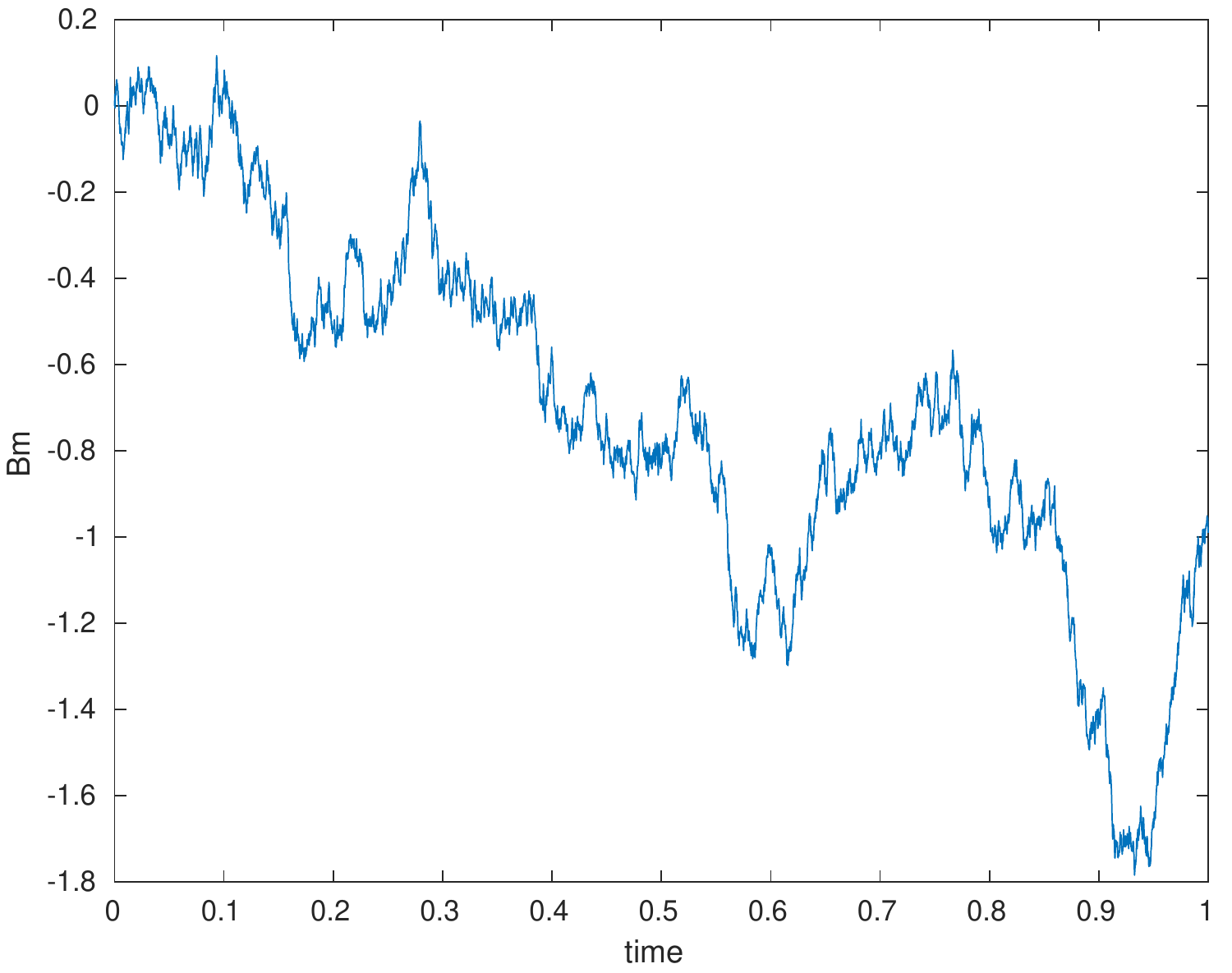}\label{fig:bm}}
  \subfloat[2. order]{\includegraphics[width=0.5\textwidth]{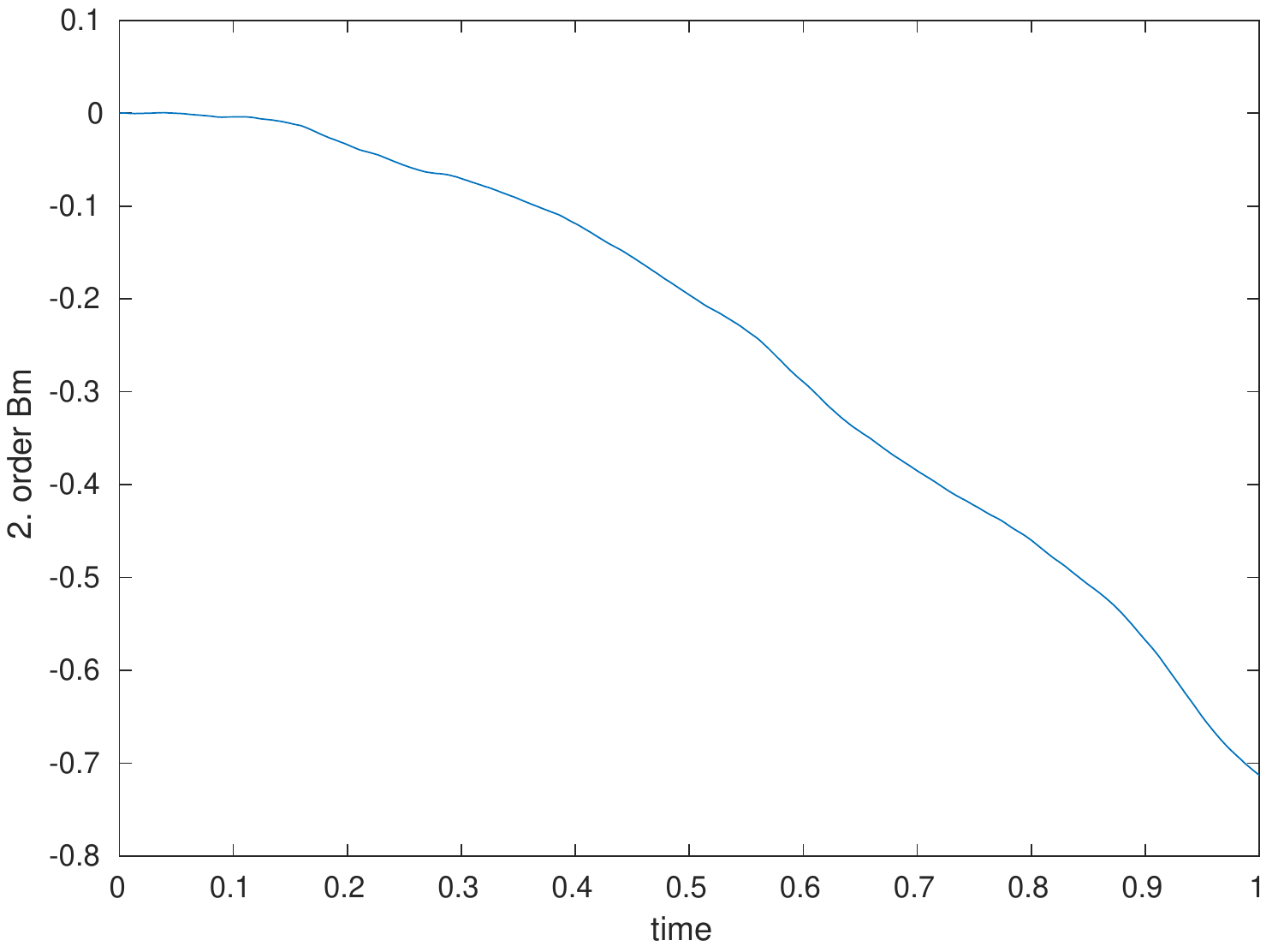}\label{fig:2bm}}\\ 
    \subfloat[3. order]{\includegraphics[width=0.5\textwidth]{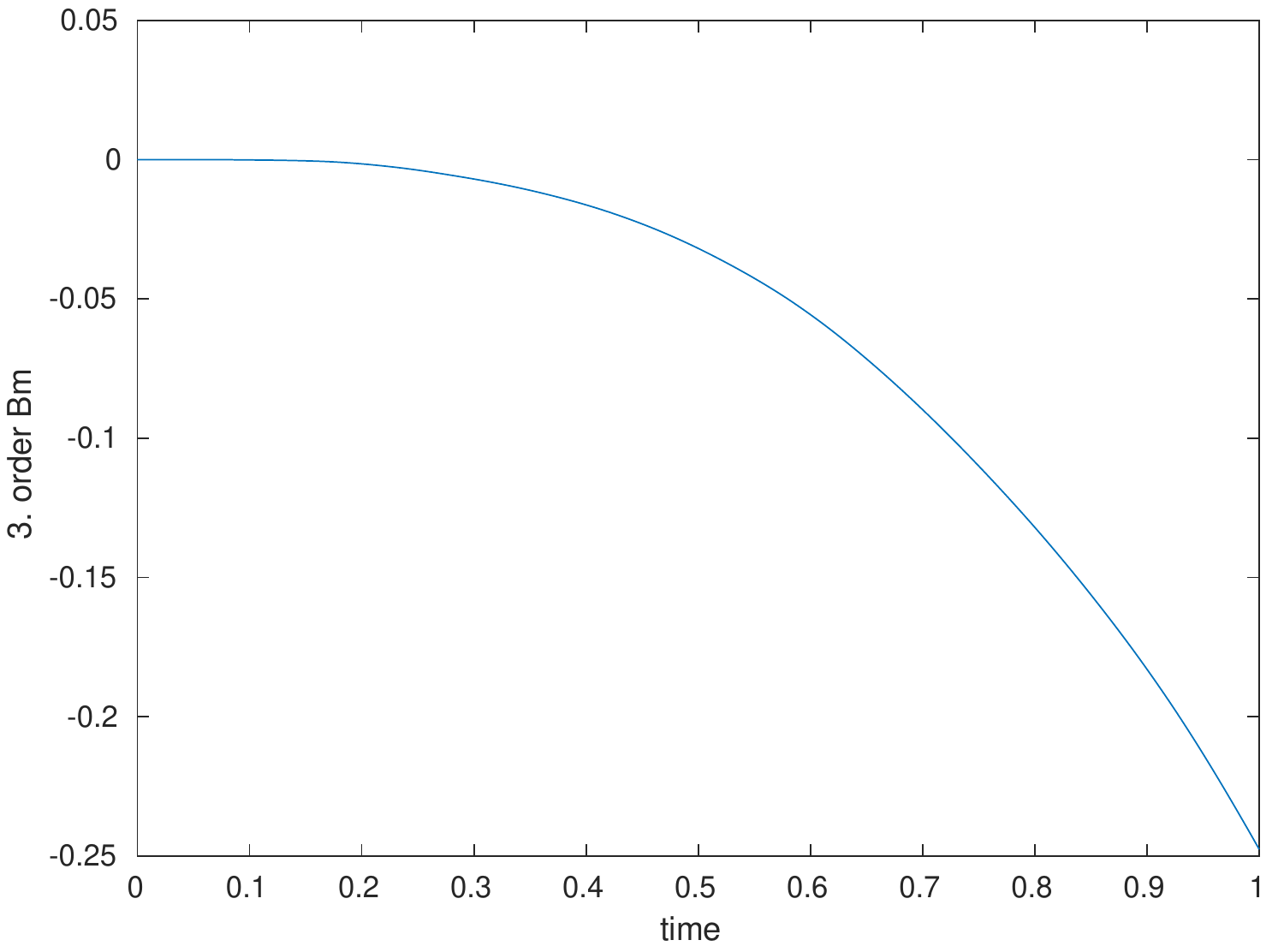}\label{fig:3bm}}
  \subfloat[4. order.]{\includegraphics[width=0.5\textwidth]{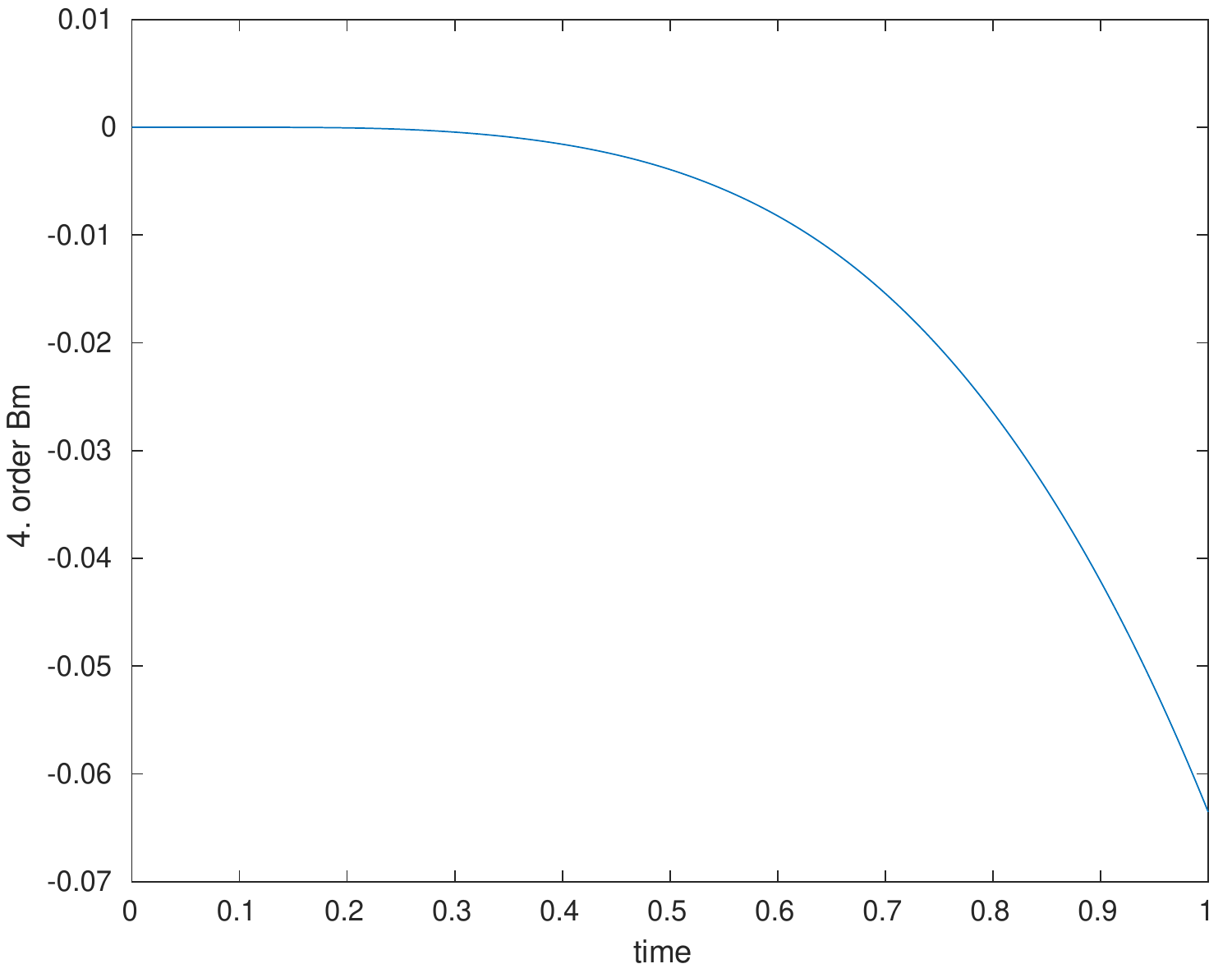}\label{fig:4bm}}\\
\end{tabular}
  \caption{1-4. order standard Brownian motion.}
\end{figure}

\begin{figure}[!htbp]
  \centering
  \begin{tabular}{cc}
  \subfloat[1. order]{\includegraphics[width=0.5\textwidth]{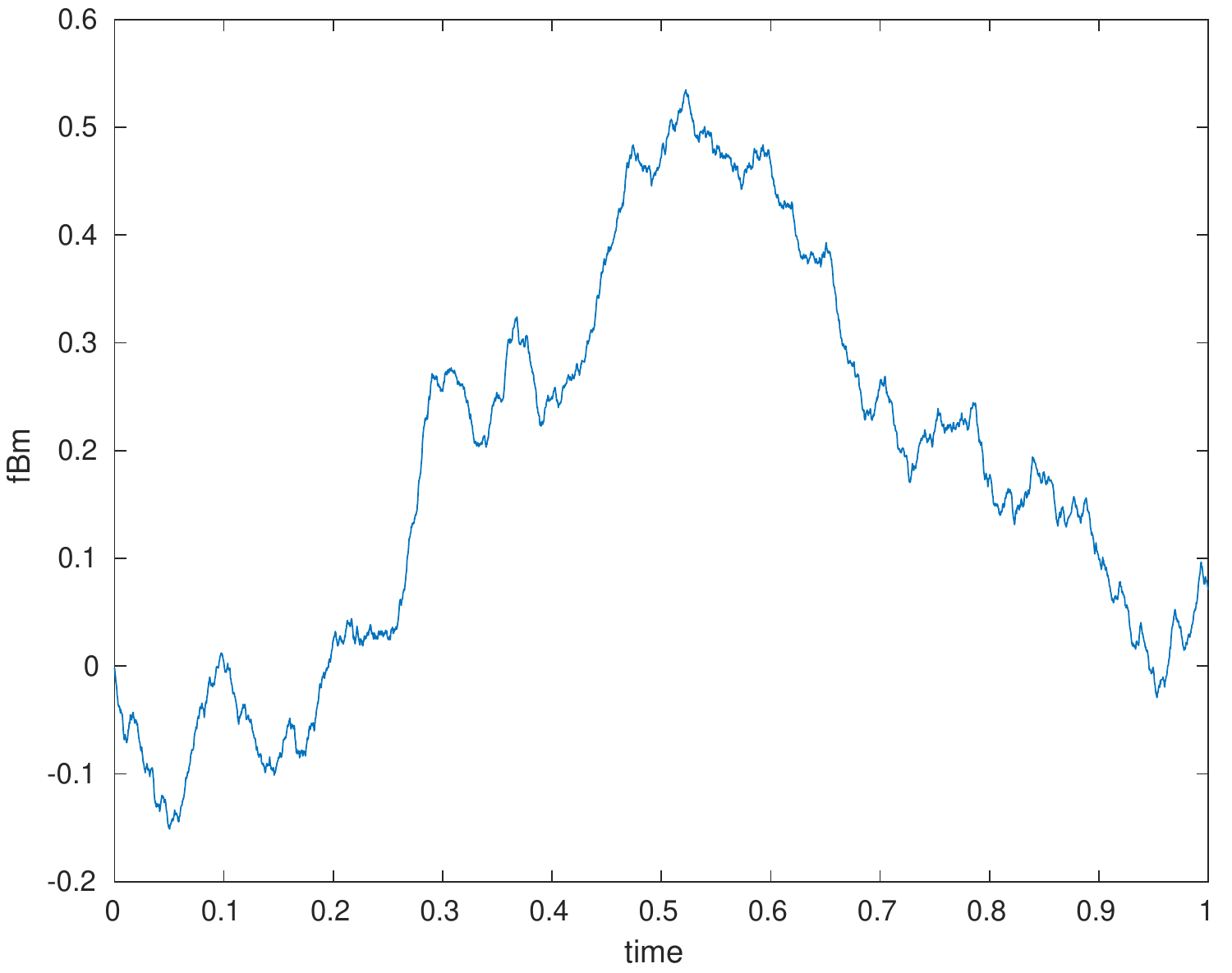}\label{fig:fbm75}}
  \subfloat[2. order]{\includegraphics[width=0.5\textwidth]{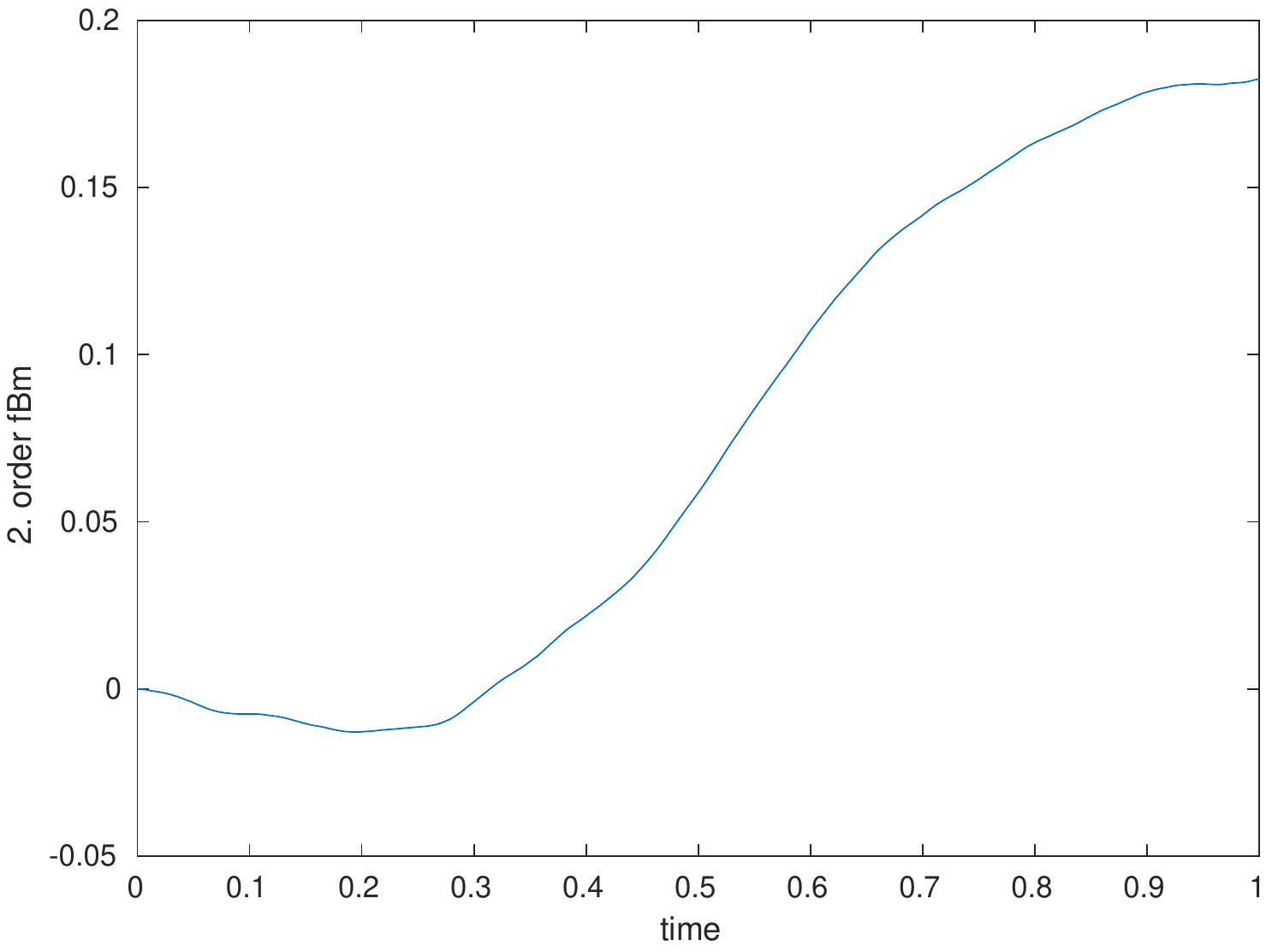}\label{fig:2fbm75}}\\ 
    \subfloat[3. order]{\includegraphics[width=0.5\textwidth]{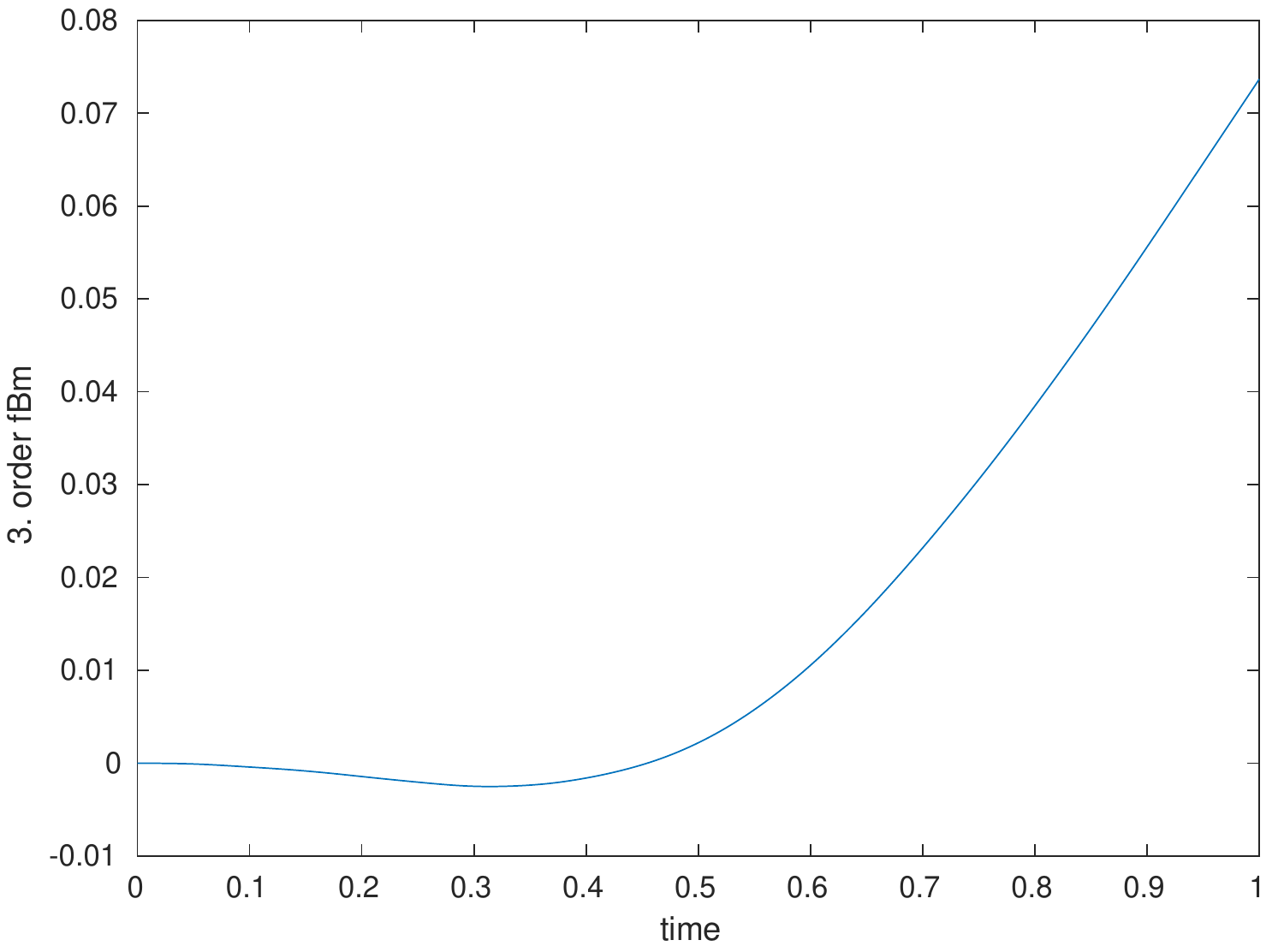}\label{fig:3fbm75}}
  \subfloat[4. order.]{\includegraphics[width=0.5\textwidth]{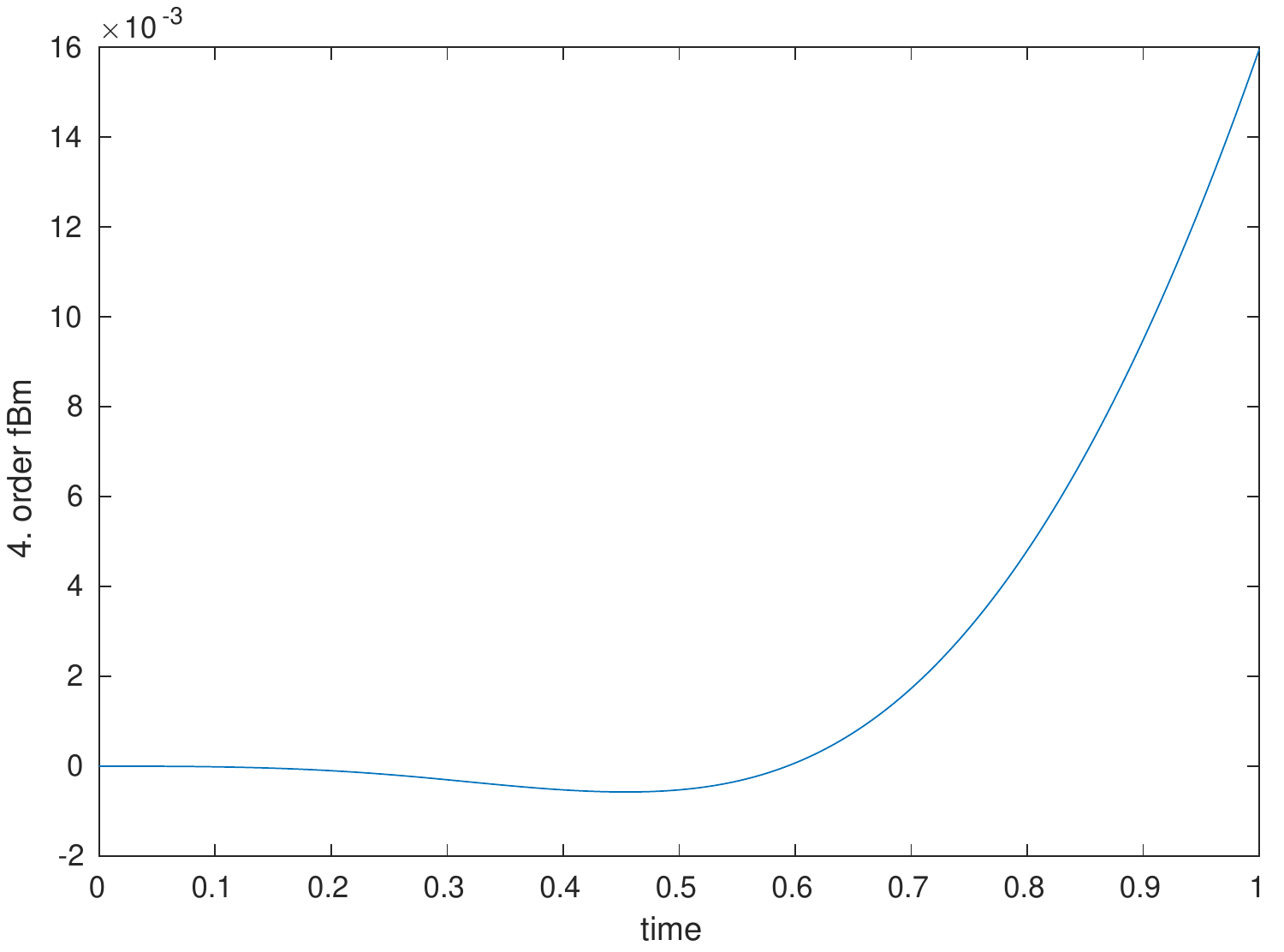}\label{fig:4fbm75}}\\
\end{tabular}
  \caption{1-4. order fractional Brownian motion with $H=0.75$.}
\end{figure}

\begin{figure}[!htbp]
  \centering
  \begin{tabular}{cc}
  \subfloat[1. order]{\includegraphics[width=0.5\textwidth]{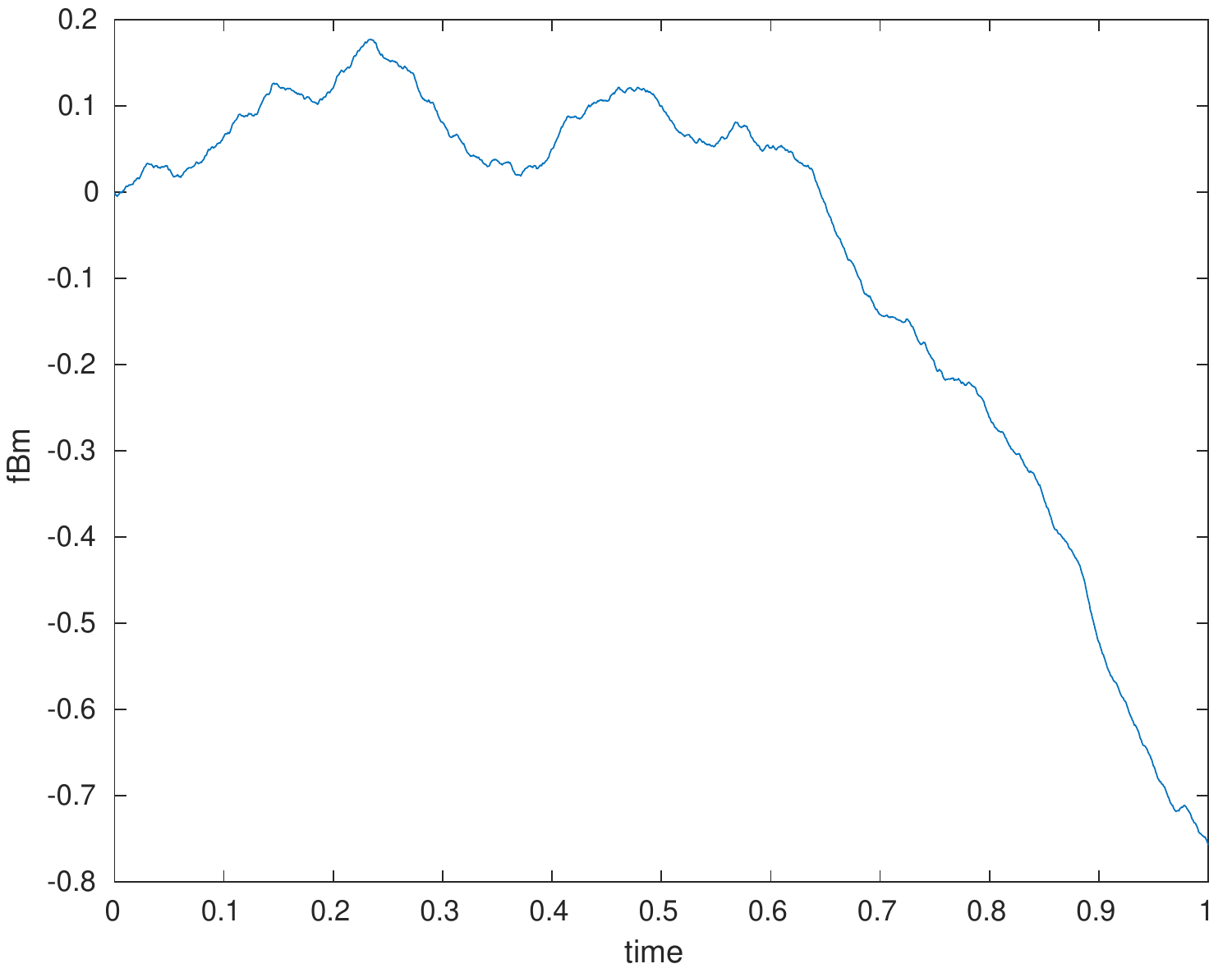}\label{fig:fbm9}}
  \subfloat[2. order]{\includegraphics[width=0.5\textwidth]{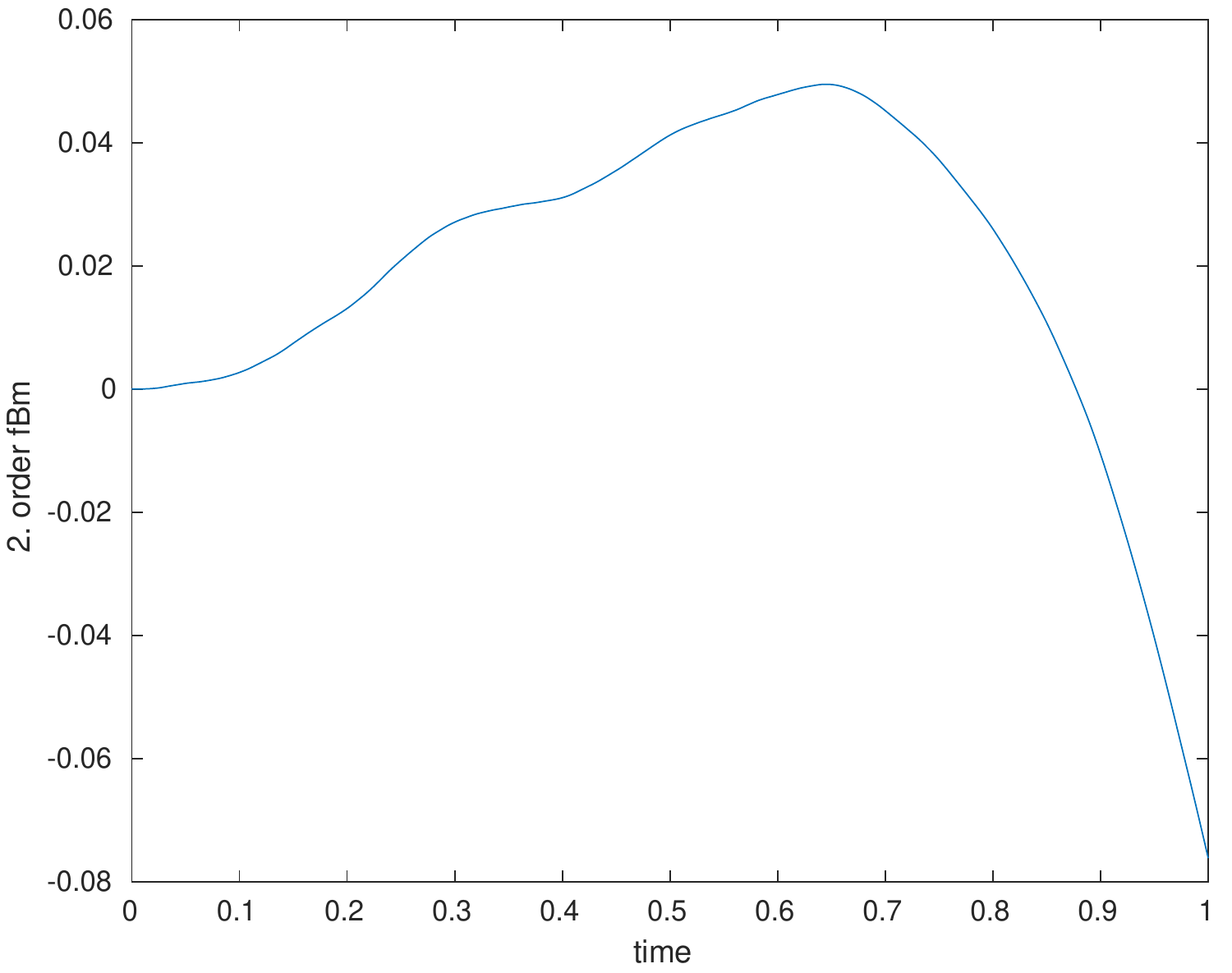}\label{fig:2fbm9}}\\ 
    \subfloat[3. order]{\includegraphics[width=0.5\textwidth]{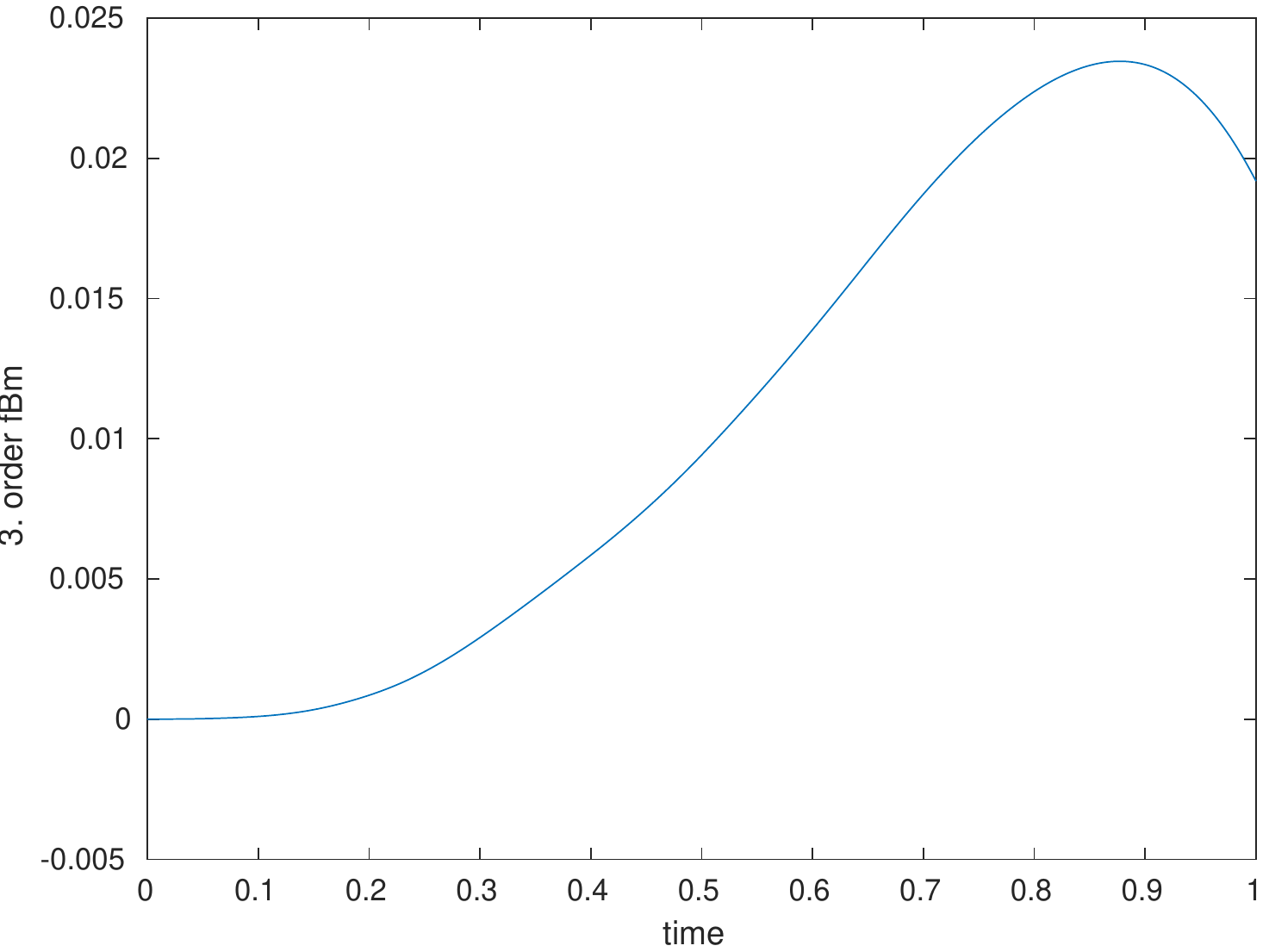}\label{fig:3fbm9}}
  \subfloat[4. order.]{\includegraphics[width=0.5\textwidth]{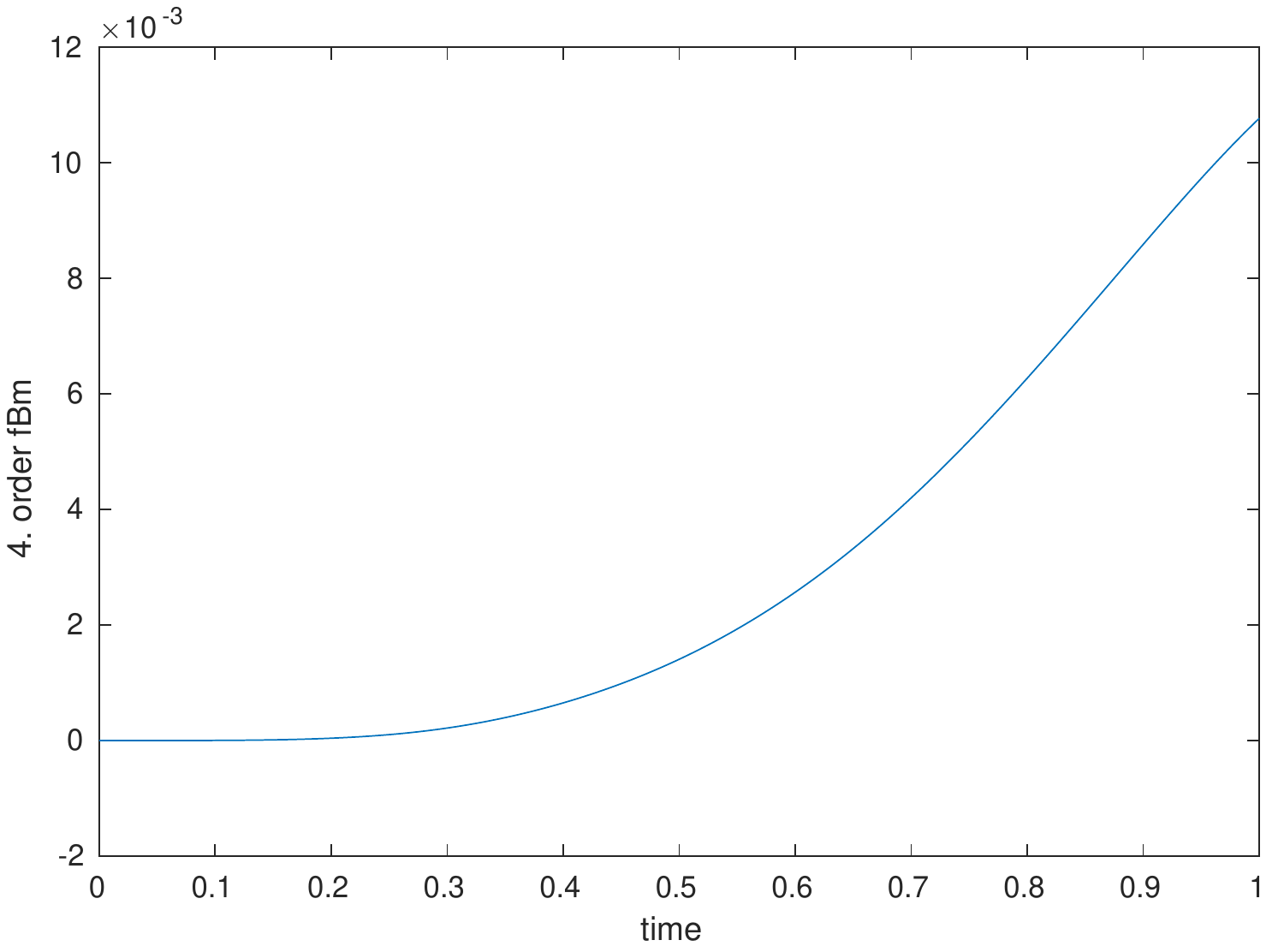}\label{fig:4fbm9}}\\
\end{tabular}
  \caption{1-4. order fractional Brownian motion with $H=0.9$.}
\end{figure}

\section{Concluding remarks}\label{sec:conclusion}
Perrin et al. \cite{Perrin-Harba-Berzin-Joseph-Iribarren-Bonami-2001} defined the $n$th order fractional Brownian motion by using the Mandelbrot--Van Ness representation \eqref{eq:mvn}, which lead to the formula \eqref{eq:mvn-nfbm}. In this article we have defined the $n$th order fractional Brownian motion $B_H^{(n)}$ of Hurst index $H\in(n-1,n)$ by using the Molchan--Golosov representation of the fractional Brownian motion. In addition, we have provided the transfer principle for $B_H^{(n)}$ that can be used to develop stochastic calculus with respect to $B_H^{(n)}$ in a relatively simple manner. In addition, our compact interval representation is very useful, e.g. for simulations. Finally, for filtering problems it is of utmost important that the filtrations of $B_H^{(n)}$ and the corresponding Brownian motion $W$ coincide. We have shown here that by using the Molchan--Golosov representation, this property is inherited directly from the same property of the underlying fractional Brownian motion. We have also used this fact to study prediction law of $B_H^{(n)}$ and the equivalence of law problem that is important for maximum likelihood estimation.


\bibliographystyle{siam}
\bibliography{pipliateekki}
\end{document}